\documentclass[reqno,11pt]{amsart}
\usepackage[foot]{amsaddr}
\usepackage{bbold}
\usepackage{charter}
\usepackage[margin=1in]{geometry}
\usepackage[colorlinks=true,linktoc=all,linkcolor=blue,citecolor=blue]{hyperref}

\theoremstyle{plain}
\newtheorem{theorem}{Theorem}[section]
\newtheorem*{theorem*}{Theorem}

\newtheorem{lemma}[theorem]{Lemma}
\newtheorem{corollary}[theorem]{Corollary}
\theoremstyle{definition}

\newtheorem*{definition*}{Definition}

\numberwithin{equation}{section}
\everymath{\displaystyle}
\allowdisplaybreaks

\newcommand{\C}{\mathbb{C}}
\newcommand{\D}{\mathbb{D}}
\newcommand{\N}{\mathbb{N}}

\newcommand{\Lip}{\mathcal{L}}
\newcommand{\Bloch}{\mathcal{B}}

\renewcommand{\mod}[1]{\left|#1\right|}

\def \X {\mathcal{X}}
\def \Y {\mathcal{Y}}
\def\a{\alpha}

\def\e{\varepsilon}

\def\g{\gamma}

\def\o{\omega}

\def\s{\sigma}
\def\t{\tau}

\def\D{\mathbb D}

\def\O{\Omega}

\def\>{\geq}

\def\ben{\begin{eqnarray}}
	\def\eeqn{\end{eqnarray}}

\title[$M_\psi$ between Lipschitz-type spaces]
{Multiplication Operators between Lipschitz-Type Spaces on a Tree}

\author{Robert F.~Allen\textsuperscript{1}, Flavia Colonna\textsuperscript{2}, and Glenn R.~Easley\textsuperscript{3}}
\address{\textsuperscript{1}Department of Mathematics, University of Wisconsin--La Crosse}
\address{\textsuperscript{2}Department of Mathematical Sciences, George Mason University}
\address{\textsuperscript{3}System Planning Corporation}
\email{allen.rob3@uwlax.edu, fcolonna@gmu.edu, geasley@sysplan.com}


\begin{document}
\begin{abstract} Let $\Lip$ be the space of complex-valued functions $f$ on the set of vertices $T$ of an rooted infinite tree rooted at $o$ such that the difference of the values of $f$ at neighboring vertices remains bounded throughout the tree, and let $\Lip_{\textbf{w}}$ be the set of functions $f\in \Lip$ such that $|f(v)-f(v^-)|=O(|v|^{-1})$, where $|v|$ is the distance between $o$ and $v$ and $v^-$ is the neighbor of $v$ closest to $o$. In this article, we characterize the bounded and the compact multiplication operators between $\Lip$ and $\Lip_{\textbf{w}}$, and provide operator norm and essential norm estimates. Furthermore, we characterize the bounded and compact multiplication operators between $\Lip_{\textbf{w}}$ and the space $L^\infty$ of bounded functions on $T$ and determine their operator norm and their essential norm. We establish that there are no isometries among the multiplication operators between these spaces.
\end{abstract}

\maketitle

\section{Introduction}
Let $\X$ and $\Y$ be complex Banach spaces of functions defined on a set $\O$.  For a complex-valued function $\psi$ defined on $\O$, the \textit{multiplication operator with symbol $\psi$ from $\X$ to $\Y$} is defined as
$$M_\psi f = \psi f, \hbox{ for all }f \in \X.$$  A fundamental objective in the study of the operators with symbol is to tie the properties of the operator to the function-theoretic properties of the symbol.

When $\O$ is taken to be the open unit disk $\D$ in the complex plane, an important space of functions to study is the {\textit{Bloch space}}, defined as the set $\Bloch$ of the analytic functions $f:\D \to \C$ for which $$\beta_f = \sup_{z \in \D} (1-\mod{z}^2)\mod{f'(z)} < \infty.$$

	The Bloch space can also be described as the set consisting of the Lipschitz functions between metric spaces from $\D$ endowed with the Poincar\'{e} distance $\rho$ to $\C$ endowed with the Euclidean distance, a fact that was proved by the second author in \cite{Colonna:89} (see also \cite{Zhu}). In fact, $f\in \mathcal{B}$ if and only if there exist $\beta > 0$ such that $$\mod{f(z)-f(w)} \leq \beta\rho(z,w),$$ and $$\beta_f = \sup_{z \neq w} \frac{\mod{f(z)-f(w)}}{\rho(z,w)}.$$

More recently, considerable research has been carried out in the field of operator theory when the set $\O$ is taken to be a discrete structure, such as a discrete group or a graph. In this work, we consider the case when $\O$ is taken to be an infinite tree.

By a \textit{tree\/} $T$ we mean a locally finite, connected, and simply-connected graph, which, as a set, we identify with the collection of its vertices. Two vertices
$u$ and $v$ are called \textit{neighbors} if there is an edge 
connecting them, and we use the notation $u\sim v$. A vertex is called \textit{terminal} if it has a
unique neighbor. A \textit{path} is a finite or infinite sequence of vertices $[v_0,v_1,\dots]$ such
that $v_k\sim v_{k+1}$ and  $v_{k-1}\ne v_{k+1}$, for all $k$.

Given a tree $T$ rooted at $o$ and a vertex $u\in T$, a vertex $v$ is called a \textit{descendant} of $u$ if $u$ lies in the unique path from $o$ to $v$. The vertex $u$ is then called an \textit{ancestor} of $v$.  Given a vertex $v\ne o$, we denote by $v^-$ the unique neighbor which is an ancestor of $v$. 
 For $v\in T$, The set $S_v$ consisting of $v$ and all its descendants is called the \textit{sector} determined by $v$.

Define the \textit{length} of a finite path $[u=u_0,u_1,\dots,v=u_n]$ (with $u_k\sim u_{k+1}$ for $k=0,\dots, n$) to be the number $n$ of edges connecting $u$ to $v$. The \textit{distance}, $d(u,v)$, between vertices $u$ and $v$ is the length of the path connecting $u$ to $v$.  The tree $T$ is a metric space under the distance $d$. Fixing $o$ as the root of the tree, we define the \textit{length} of a vertex $v$, by $|v|=d(o,v)$.  By a \textit{function on a tree} we mean a complex-valued function on the set of its vertices.

In this paper, the tree will be assumed to be rooted at a vertex $o$ and without terminal vertices (and hence infinite).

Infinite trees are discrete structures which exhibit significant geometric and potential-theoretic characteristics that are present in the Poincar\'e disk $\D$. For instance, they have a boundary, which is defined as the set of equivalence classes of paths which differ by finitely many vertices. The union of the boundary with the tree yields a compact space. A useful resource for the potential theory on trees illustrating the commonalities with the disk is \cite{Cartier}. In \cite{CohenColonna:94} it was shown that, if the tree has the property that all its vertices have the same number of neighbors, then there is a natural embedding of the tree in the unit disk such that the edges of the tree are arcs of geodesics in $\D$ with the same hyperbolic length and the set of cluster points of the vertices is the entire unit circle.

In \cite{ColonnaEasley:10}, the last two authors defined the \textit{Lipschitz space} $\Lip$ on a tree $T$ as the set consisting of the functions $f:T \to \C$ which are Lipschitz with respect to the distance $d$ on $T$ and the Euclidean distance on $\C$.  For this reason, the Lipschitz space $\Lip$ can be viewed as a discrete analogue of the Bloch space $\Bloch$.  It was also shown that the Lipschitz functions on $T$ are precisely the functions for which $$\|Df\|_\infty = \sup_{v \in T^*} Df(v) < \infty,$$ where $Df(v) = |f(v)-f(v^-)|$ and  $T^* = T\setminus\{o\}$.  Under the norm $$\|f\|_{\Lip} = |f(o)| + \|Df\|_\infty,$$ $\Lip$ is a Banach space containing the space $L^\infty$ of the bounded functions on $T$. Furthermore, for $f\in L^\infty$, $\|f\|_\Lip\le 2\|f\|_\infty$.

 The \textit{little Lipschitz space} is defined as $$\Lip_0=\left\{f\in\Lip: \lim_{|v|\to\infty}Df(v)=0\right\},$$
and was proven to be a separable closed subspace of $\Lip$. We state the following results that will be useful in the present work.

\begin{lemma}[Lemma 3.4 of \cite{ColonnaEasley:10}]\label{old} $\text{}$
\begin{enumerate}
\item[{\rm{(a)}}] If $f \in \Lip$ and $v \in T$, then $$|f(v)| \leq |f(o)| + |v|\\ \|Df\|_\infty.$$ In particular, if $\|f\|_\Lip\le 1$, then $|f(v)|\le |v|$ for each $v\in T^*$.
\smallskip

\item[{\rm{(b)}}] If $f\in\Lip_0$, then $$\lim_{|v|\to\infty}\frac{f(v)}{|v|}=0.$$
\end{enumerate}
\end{lemma}

\begin{lemma}[Proposition 2.4 of \cite{ColonnaEasley:10}]\label{weakconv_Lip} Let $\{f_n\}$ be a sequence of functions in $\Lip_0$ converging to $0$ pointwise in $T$ and such that $\{\|f_n\|_\Lip\}$ is bounded. Then $f_n\to 0$ weakly in $\Lip_0$.
\end{lemma}

In \cite{AllenColonnaEasley:10}, we introduced the \textit{weighted Lipschitz space}  on a tree $T$ as the set $\Lip_{\textbf{w}}$ of the functions $f:T \to \C$ such that $$\sup_{v \in T^*} |v|Df(v) < \infty.$$  The interest in this space is due to its connection to the bounded multiplication operators on $\Lip$. Specifically, it was shown in \cite{ColonnaEasley:10} that the bounded multiplication operators on $\Lip$ are precisely those operators $M_\psi$ whose symbol $\psi$ is a bounded function in $\Lip_{\textbf{w}}$. The space $\Lip_{\textbf{w}}$ was shown to be a Banach space under the norm $$\|f\|_{\textbf{w}} = |f(o)| + \sup_{v \in T^*} |v|Df(v).$$

	The \textit{little weighted Lipschitz space} was defined as $$\Lip_{\textbf{w},0} = \left\{f\in\Lip_{\textbf{w}} : \lim_{|v|\to\infty}|v|Df(v)=0\right\},$$ and was shown to be a closed separable subspace of $\Lip_{\textbf{w}}$.

In this paper, we shall make repeated use of the following results proved in \cite{AllenColonnaEasley:10}.

\begin{lemma}[Propositions 2.1 and 2.6 of \cite{AllenColonnaEasley:10}]\label{new} $\text{}$
\begin{enumerate}
\item[{\rm{(a)}}] If {\rm$f \in \Lip_{\textbf{w}}$}, and $v\in T^*$, then {\rm$$|f(v)| \leq (1+\log|v|)\|f\|_{\textbf{w}}.$$}

\item[{\rm{(b)}}] If {\rm$f\in\Lip_{{\textbf{w}},0}$}, then $$\lim_{|v|\to\infty}\frac{f(v)}{\log|v|}=0.$$
\end{enumerate}
\end{lemma}

\begin{lemma}[Proposition 2.7 of \cite{AllenColonnaEasley:10}]\label{weakconv_Lipw} Let $\{f_n\}$ be a sequence of functions in {\rm $\Lip_{{\textbf{w}},0}$} converging to $0$ pointwise in $T$ and such that {\rm $\{\|f_n\|_{\textbf{w}}\}$} is bounded. Then $f_n\to 0$ weakly in {\rm $\Lip_{{\textbf{w}},0}$}.
\end{lemma}

In this paper, we consider the multiplication operators between $\Lip$ and $\Lip_{\textbf{w}}$, as well as between $\Lip_{\textbf{w}}$ and $L^\infty$.  The multiplication operators between $\Lip$ and $L^\infty$ were studied by the last two authors in \cite{ColonnaEasley:10-II}.  

\subsection{Organization of the paper}

In Sections 2 and 3, we study the multiplication operators between $\Lip_{{\textbf{w}}}$  and $\Lip$. We characterize the bounded and the compact operators, and give estimates on their operator norm and their essential norm. We also prove that no isometric multiplication operators exist between the respective spaces.

In Section 4, we characterize the bounded operators and the compact operators from $\Lip_{{\textbf{w}}}$ to $L^\infty$ and determine their operator norm and their essential norm. As was the case in Sections~2 and 3, we show that no isometries exist amongst such operators. In addition, we characterize the multiplication operators that are bounded from below.

Finally, in Section~5, we characterize the bounded and the compact multiplication operators from $L^\infty$ to $\Lip_{{\textbf{w}}}$. We also determine their operator norm and their essential norm. As with all the other cases, we show that there are no isometries amongst such operators.

\section{Multiplication operators from $\Lip_{\textbf{w}}$ to $\Lip$}

We begin the section with the study the bounded multiplication operators $M_\psi : \Lip_{\textbf{w}} \to \Lip$ and $M_\psi : \Lip_{\textbf{w},0} \to \Lip_0$. 

\subsection{Boundedness and Operator Norm Estimates}
Let $\psi$ be a function on the tree $T$.  Define
$$\begin{aligned}
\t_\psi &= \sup_{v\in T^*}D\psi(v)\log(1+|v|),\\
\s_\psi &= \sup_{v\in T}\frac{|\psi(v)|}{|v|+1}.
\end{aligned}$$
In the following theorem, we give a boundedness criterion in terms of the quantities $\t_\psi$ and $\s_\psi$.

\begin{theorem}\label{lipmtolip} For a function $\psi$ on $T$, the following statements are equivalent:
\begin{enumerate}
\item[{\rm{(a)}}] {\rm $M_\psi:\Lip_{\textbf{w}}\to \Lip$} is bounded.
\item[{\rm{(b)}}] {\rm $M_\psi:\Lip_{{\textbf{w}},0}\to \Lip_0$} is bounded.
\item[{\rm{(c)}}] $\t_\psi$ and $\s_\psi$ are finite.
\end{enumerate}
Furthermore, under these conditions, we have
$$\max\{\t_\psi,\s_\psi\}\le \|M_\psi\|\le \t_\psi+\s_\psi.$$
\end{theorem}

\begin{proof} $(a)\Longrightarrow (c)$ Assume $M_\psi:\Lip_{\textbf{w}}\to \Lip$ is bounded. Applying $M_\psi$ to the constant function 1, we have $\psi\in\Lip$, so that, by Lemma~\ref{old}, we have $\s_\psi<\infty$. Next, consider the function $f$ on $T$ defined by $f(v)=\log(1+|v|)$. Then $f(o)=0$; for $v\ne o$, a straightforward calculation shows that
$$|v|Df(v)=|v|(\log(1+|v|)-\log|v|)\le 1$$
and $\lim\limits_{|v|\to\infty}|v|Df(v)=1$. Thus, $\|f\|_{\textbf{w}}=1$ and so $\|M_\psi f\|_\Lip\le \|M_\psi\|$. Therefore, for $v\in T^*$, noting that
$$D(\psi f)(v)=D\psi(v)f(v)+\psi(v^-)Df(v),$$ we have
\ben D\psi(v)|f(v)|&\le& D(\psi f)(v)+|\psi(v^-)|Df(v)\nonumber\\
&\le &\|M_\psi f\|_\Lip+\s_\psi|v|Df(v)\le \|M_\psi\|+\s_\psi.\nonumber\eeqn
Hence $\t_\psi<\infty.$

$(c)\Longrightarrow (a)$ Assume $\t_\psi$ and $\s_\psi$ are finite. Then, by Lemma~\ref{new}, for $f\in\Lip_{\textbf{w}}$ and $v\in T^*$, we have
\ben D(\psi f)(v)&\le& D\psi (v)|f(v)|+|\psi(v^-)|Df(v)\nonumber\\
&\le & D\psi(v)(1+\log|v|)\|f\|_{\textbf{w}}+|v|\s_\psi Df(v)\nonumber\\
&\le &\t_\psi\|f\|_{\textbf{w}}+\s_\psi(\|f\|_{\textbf{w}}-|f(o)|).\nonumber\eeqn
Thus, since $|\psi(o)|\le \s_\psi$, we obtain  \ben \|M_\psi f\|_\Lip&\le &|\psi(o)||f(o)|+\t_\psi\|f\|_{\textbf{w}}+\s_\psi(\|f\|_{\textbf{w}}-|f(o)|)\nonumber\\
&=&(\t_\psi+\s_\psi)\|f\|_{\textbf{w}}+(|\psi(o)|-\s_\psi)|f(o)|\nonumber\\
&\le &\left(\t_\psi+\s_\psi\right)\|f\|_{\textbf{w}},\nonumber\eeqn
proving the boundedness of $M_\psi:\Lip_{\textbf{w}}\to \Lip$ and the upper estimate.

$(b)\Longrightarrow (c)$ Suppose $M_\psi:\Lip_{{\textbf{w}},0}\to\Lip_0$ is bounded. The finiteness of $\s_\psi$ follows again from the fact that $\psi=M_\psi 1\in\Lip_0$ and from Lemma~\ref{old}. To prove that $\t_\psi<\infty$, let $0<\a<1$ and, for $v\in T$, define $f_\a(v)=(\log(1+|v|))^\a$. Then $f_\a(o)=0$ and $|v|Df_a(v)\to 0$ as $|v|\to\infty$; so $f_\a\in \Lip_{{\textbf{w}},0}$. Since for $0<\a<1$, the function $x\mapsto x-x^\a$ is increasing for $x\ge 1$, the function $Df_\a(v)$ is increasing in $\a$ and $Df_\a(v)\le Df(v)$ for $v\in T^*$, where $f(v)=\log(1+|v|)$, for $v\in T$. Thus, $\|f_\a\|_{\textbf{w}}\le\|f\|_{\textbf{w}}= 1$. Therefore, for $v\in T^*$, we have
\ben D\psi(v)|f_\a(v)|&\le &D(\psi f_\a)(v)+|\psi(v^-)|Df_\a(v)\nonumber\\
&\le &\|M_\psi f_\a\|+\s_\psi|v|D f_\a(v)\le \|M_\psi\|+\s_\psi.\nonumber\eeqn
Letting $\a\to 1$, we obtain
$$D\psi(v)\log(1+|v|)\le \|M_\psi\|+\s_\psi.$$ Hence  $\t_\psi<\infty$.

$(c)\Longrightarrow (b)$ Assume $\sigma_\psi$ and $\tau_\psi$ are finite, and let $f\in\Lip_{{\textbf{w}},0}$. Then, by Lemma~\ref{new}, for $v\in T^*$, we have
\ben D(\psi f)(v)&\le &D\psi(v)|f(v)|+|\psi(v^-)|Df(v)\nonumber\\
&\le &D\psi(v)\log(1+|v|)\frac{|f(v)|}{\log(1+|v|)}+\frac{|\psi(v^-)|}{|v|}|v|Df(v)\nonumber\\
&\le&\t_\psi\frac{|f(v)|}{\log(1+|v|)}+\s_\psi|v|Df(v)\to 0\nonumber\eeqn
as $|v|\to\infty$. Thus, $\psi f\in\Lip_0.$ The boundedness of $M_\psi$ and the estimate $\|M_\psi f\|_\Lip\le \t_\psi+\s_\psi$ can be shown as in the proof of $(c)\Longrightarrow (a)$.

Finally we show that, under boundedness assumptions on $M_\psi$, $\|M_\psi\|\ge \max\{\t_\psi,\s_\psi\}$.  For $v\in T^*$, let $f_v=\frac1{|v|+1}\chi_v$, where $\chi_v$ denotes the characteristic function of $\{v\}$. Then $\|f_v\|_{\textbf{w}}=1$ and $$\|\psi f_v\|_\Lip=\frac{|\psi(v)|}{|v|+1}.$$ Furthermore, letting $f_o=\frac12\chi_o$, we see that $\|f_o\|_{\textbf{w}}=1$ and $\|\psi f_o\|_\Lip=|\psi(o)|$. Therefore, we deduce that $\|M_\psi\|\ge \s_\psi.$

Next, fix $v\in T^*$ and for $w\in T$, define
$$g_v(w)=\begin{cases} \log(1+|w|) &\hbox{ if }|w|<|v|,\\
\log(1+|v|) &\hbox{ if }|w|\ge |v|.\end{cases}$$
Then, $g_v\in \Lip_{\textbf{w}}$ and $$\lim_{|v|\to\infty}\|g_v\|_{\textbf{w}}=\lim_{|v|\to\infty}|v|\left[\log(1+|v|)-\log|v|\right]=1.$$
Observe that for $w\in T^*$, we have
$$D(\psi g_v)(w)=\begin{cases} |\psi(w)\log(1+|w|)-\psi(w^-)\log|w|| &\hbox{ if }|w|<|v|,\\
D\psi(w)\log(1+|v|)&\hbox{ if }|w|\ge |v|.\end{cases}$$
Hence $$\sup_{w\in T^*}D(\psi g_v)(w)\ge \sup_{|w|\ge |v|}D\psi(w)\log(1+|v|)\ge D\psi(v)\log(1+|v|).$$
Define $f_v=\frac{g_v}{\|g_v\|_{\textbf{w}}}$. Then $\|f_v\|_{\textbf{w}}=1$ and
\ben \|M_\psi\|\ge \|M_\psi f_v\|_\Lip= \frac{\|D(\psi g_v)\|_\infty}{\|g_v\|_{\textbf{w}}}\ge \frac{D\psi(v)\log(1+|v|)}{\|g_v\|_{\textbf{w}}}.\nonumber\eeqn
Taking the limit as $|v|\to\infty$, we obtain $\|M_\psi\|\ge \t_\psi.$
Therefore, $\|M_\psi\|\ge \max\{\t_\psi,\s_\psi\}.$
\end{proof}

\subsection{Isometries}

In this section, we show there are no isometric multiplication operators $M_\psi$ from the spaces $\Lip_{\textbf{w}}$ or $\Lip_{\textbf{w},0}$ to the spaces $\Lip$ or $\Lip_0$.

Assume $M_\psi:\Lip_{\textbf{w}}\to\Lip$ is an isometry. Then
$\|\psi\|_\Lip=\|M_\psi 1\|_\Lip=1.$ On the other hand, $|\psi(o)|=\frac12\left\|M_\psi\chi_o\right\|_\Lip=\frac12\left\|\chi_o\right\|_{\textbf{w}}=1.$
Thus $\sup\limits_{v\in T^*}D\psi(v)=\|\psi\|_\Lip-|\psi(o)|=0$, which implies that $\psi$ is a constant of modulus 1. Yet, for $v\in T^*$, letting $f_v=\frac1{|v|+1}\chi_v$, we see that $$1=\|f_v\|_{\textbf{w}}=\|M_\psi f_v\|_\Lip=\frac1{|v|+1},$$
which yields a contradiction. Therefore, we obtain the following result.

\begin{theorem}\label{noiso2} There are no isometries $M_\psi$ from {\rm $\Lip_{\textbf{w}}$} to $\Lip$ or {\rm $\Lip_{{\textbf{w}},0}$} to $\Lip_0$.
\end{theorem}

\subsection{Compactness and Essential Norm Estimates}
In this section, we characterize the compact multiplication operators.  As with many classical spaces, the characterization of the compact operators is a ``little-oh" condition corresponding the the ``big-oh" condition for boundedness.  We first collect some useful results about compact operators from $\Lip_{\textbf{w}}$ or from $\Lip_{\textbf{w},0}$ to $\Lip$.

\begin{lemma}\label{compact lemma section 2} A bounded multiplication operator $M_\psi$ from {\rm $\mathcal{L}_{\textbf{w}}$} to $\mathcal{L}$ is compact if and only if for every bounded sequence $\{f_n\}$ in $\mathcal{L}_w$ converging to 0 pointwise, the sequence $\{\|\psi f_n\|_\mathcal{L}\} \to 0$ as $n \to \infty$.
\end{lemma}

\begin{proof} Assume $M_\psi$ is compact, and let $\{f_n\}$ be a bounded sequence in $\mathcal{L}_{\textbf{w}}$ converging to 0 pointwise.  Without loss of generality, we may assume $\|f_n\|_{\textbf{w}} \leq 1$ for all $n \in \N$.  Then the sequence $\{M_\psi f_n\} = \{\psi f_n\}$ has a subsequence $\{\psi f_{n_k}\}$ which converges in the $\mathcal{L}$-norm to some function $f \in \mathcal{L}$.  Clearly $\psi(o)f_{n_k}(o) \to \psi(o)f(o)$, and by part (a) of Lemma~\ref{old}, for $v \in T^*$, we have
$$\begin{aligned}
|\psi(v)f_{n_k}(v) - f(v)| &\leq |\psi(o)f_{n_k}(o) - f(o)| + |v|\|D(\psi f_{n_k} - f)\|_\infty\\
&\leq (1+|v|)\|\psi f_{n_k} - f\|_\mathcal{L}.
\end{aligned}$$  Thus, $\psi f_{n_k} \to f$ pointwise on $T$.  Since $f_n \to 0$ pointwise, it follows that $f$ must be identically 0, which implies that $\|\psi f_{n_k}\|_\mathcal{L} \to 0$.  With 0 being the only limit point of $\{\psi f_n\}$ in $\mathcal{L}$, it follows that $\|\psi f_n\|_\Lip\to 0$ as $n \to \infty$.

Conversely, assume every bounded sequence $\{f_n\}$ in $\mathcal{L}_{\textbf{w}}$ converging to 0 pointwise has the property that $\|\psi f_n\|_\mathcal{L} \to 0$ as $n \to \infty$.  Let $\{g_n\}$ be a sequence in $\mathcal{L}_{\textbf{w}}$ with $\|g_n\|_{\textbf{w}} \leq 1$ for all $n \in \N$.  Then $|g_n(o)| \leq 1$ for all $n \in \N$, and by part (a) of Lemma 1.2, for $v \in T^*$, we obtain $$|g_n(v)| \leq (1+\log|v|)\|g_n\|_{\textbf{w}} \leq 1+\log|v|.$$  Thus, $\{g_n\}$ is uniformly bounded on finite subsets of $T$.  So some subsequence $\{g_{n_k}\}$ converges pointwise to some function $g$. Fix $v\in T^*$ and $\varepsilon > 0$. Then for $k$ sufficiently large, we have
$$|g(v)-g_{n_k}(v)| < \frac{\varepsilon}{2|v|}, \text{ and } |g_{n_k}(v^-)-g(v^-)| < \frac{\varepsilon}{2|v|}.$$ We deduce
$$\begin{aligned}
|v|Dg(v) &\leq |v||g(v) - g_{n_k}(v) + g_{n_k}(v^-) - g(v^-)| + |v|Dg_{n_k}(v)\\
&\leq  |v||g(v)-g_{n_k}(v)| + |v||g_{n_k}(v^-)-g(v^-)| + |v|Dg_{n_k}(v)\\
&< \varepsilon + |v|Dg_{n_k}(v)\le \varepsilon + 1,
\end{aligned}$$ for all $k$ sufficiently large. So $g \in \mathcal{L}_{\textbf{w}}$.  The sequence defined by $f_k = g_{n_k}-g$ is bounded in $\mathcal{L}_{\textbf{w}}$ and converges to 0 pointwise.  Thus by hypothesis, we obtain $\|\psi f_k\|_\mathcal{L} \to 0$ as $k \to \infty$.  It follows that $M_\psi g_{n_k} = \psi g_{n_k} \to \psi g$ in the $\mathcal{L}$-norm, thus proving the compactness of $M_\psi$.\;\end{proof}

By an analogous argument, we obtain the corresponding compactness criterion for $M_\psi$ from {\rm $\mathcal{L}_{{\textbf{w}},0}$} to $\mathcal{L}_0$.

\begin{lemma}\label{compact lemma section 2_0} A bounded multiplication operator $M_\psi$ from {\rm $\mathcal{L}_{{\textbf{w}},0}$} to $\mathcal{L}_0$ is compact if and only if for every bounded sequence $\{f_n\}$ in {\rm $\mathcal{L}_{{\textbf{w}},0}$} converging to 0 pointwise, the sequence $\{\|\psi f_n\|_\mathcal{L}\} \to 0$ as $n \to \infty$.
\end{lemma}

	The following result is a variant of Lemma~\ref{new}(a), which will be needed to prove a characterization of the compact multiplication operators from $\mathcal{L}_{\textbf{w}}$ to $\Lip$ and from $\mathcal{L}_{{\textbf{w}},0}$ to $\mathcal{L}_0$ (Theorem~\ref{compactness2}).

\begin{lemma}\label{variantnew} For {\rm $f\in\Lip_{\textbf{w}}$} and $v\in T$ {\rm
\ben |f(v)|\le |f(o)|+2\log(1+|v|)s_{\textbf{w}}(f),\label{estnew}\eeqn
where $s_{\textbf{w}}(f)=\sup\limits_{w\in T^*}|w|Df(w).$}
\end{lemma}

\begin{proof} Fix $v\in T$ and argue by induction on $n=|v|.$ For $n=0$, the inequality (\ref{estnew}) is obvious. So assume $|v|=n>0$ and $|f(u)|\le |f(o)|+2\log(1+|u|)s_{\textbf{w}}(f)$ for all vertices $u$ such that $|u|<n$. Then
\ben |f(v)|&\le&|f(v)-f(v^-)|+|f(v^-)|\nonumber\\
&\le&\frac1{|v|}s_{\textbf{w}}(f)+|f(o)|+2\log|v|s_{\textbf{w}}(f)\nonumber\\
&=&|f(o)|+\left(\frac1{|v|}+2\log|v|\right)s_{\textbf{w}}(f).\label{induct}\eeqn
Next, observe that $\frac1{|v|+1}\le \log\left(\frac{|v|+1}{|v|}\right)$, so
$$\frac1{|v|}\le \frac2{|v|+1}\le2\log\left(\frac{|v|+1}{|v|}\right).$$
Hence \ben \frac1{|v|}+2\log|v|\le 2\log(|v|+1).\label{logest}\eeqn Inequality (\ref{estnew}) now follows immediately from (\ref{induct}) and (\ref{logest}).
\end{proof}

\begin{theorem}\label{compactness2} Let $M_\psi$ be a bounded multiplication operator from {\rm $\mathcal{L}_{\textbf{w}}$} to $\mathcal{L}$ (or equivalently from {\rm $\mathcal{L}_{{\textbf{w}},0}$} to $\mathcal{L}_0$).  Then the following statements are equivalent:
\begin{enumerate}
\item[(a)] {\rm $M_\psi:\mathcal{L}_{\textbf{w}} \to \mathcal{L}$} is compact.
\item[(b)] {\rm $M_\psi:\mathcal{L}_{{\textbf{w}},0} \to \mathcal{L}_0$} is compact.
\item[(c)] $\displaystyle\lim_{|v|\to\infty} \frac{|\psi(v)|}{|v|+1} = 0$ and $\displaystyle\lim_{|v|\to\infty} D\psi(v)\log|v| = 0$.
\end{enumerate}
\end{theorem}

\begin{proof} We first prove $(a) \Longrightarrow (c)$.  Assume $M_\psi:\mathcal{L}_{\textbf{w}} \to \mathcal{L}$ is compact. It suffices to show that for any sequence $\{v_n\}$ in $T$ such that $2 \le |v_n| \to \infty$, we have
$\displaystyle\lim_{n \to \infty} \frac{|\psi(v_n)|}{|v_n|+1} = 0$ and $\displaystyle\lim_{n \to \infty}D\psi(v_n)\log|v_n| = 0$.
Let $\{v_n\}$ be such a sequence, and for each $n \in \N$, define $f_n = \frac{1}{|v_n|+1}\chi_{v_n}$.  Then $f_n(o) = 0$, $f_n \to 0$ pointwise as $n \to \infty$, and $\|f_n\|_{\textbf{w}} = 1$. By Lemma~\ref{compact lemma section 2}, it follows that $\|\psi f_n\|_\Lip\to 0$ as $n\to\infty$. Furthermore
$$\begin{aligned}
\|\psi f_n\|_\mathcal{L} = \sup_{v \in T^*} |\psi(v)f_n(v) - \psi(v^-)f_n(v^-)|
= |\psi(v_n)f_n(v_n)| = \frac{|\psi(v_n)|}{|v_n|+1}.
\end{aligned}$$ Thus $\lim\limits_{n \to \infty} \displaystyle\frac{|\psi(v_n)|}{|v_n|+1} = 0.$

Next, for each $n \in \N$ and $v\in T$, define
$$g_n(v) = \begin{cases}
0 &\text{ if } |v| < \sqrt{|v_n|},\\
2\log|v| - \log|v_n| &\text{ if } \sqrt{|v_n|} \leq |v| < |v_n|-1,\\
\log|v_n| &\text{ if } |v| \geq |v_n|-1.
\end{cases}$$  Then $Dg_n(v) = 0\,$ if $|v| \le \sqrt{|v_n|}\,$ or $|v| > |v_n|-1$. In addition, if $\sqrt{|v_n|} < |v| \le |v_n|-1$, then $|v|Dg_n(v) < 4$. Indeed, there are two possibilities. Either $\sqrt{|v_n|} \le |v|-1$, in which case $$|v|Dg_n(v)=2|v||(\log|v|-\log(|v|-1))\le \frac{2|v|}{|v|-1}\le 3,$$ or $|v|-1<\sqrt{|v_n|}<|v|$, in which case \ben |v|Dg_n(v)&=&|v|(2\log|v|-\log|v_n|)\nonumber\\&\le &(\sqrt{|v_n|}+1)\log\frac{(\sqrt{|v_n|}+1)^2}{|v_n|}\nonumber\\&\le& \frac{2(\sqrt{|v_n|}+1)}{\sqrt{|v_n|}}\le 2\left(1+\frac1{\sqrt{2}}\right)<4.\nonumber\eeqn
 Thus $\{\|g_n\|_{\textbf{w}}\}$ is bounded and $\{g_n\}$ converges to 0 pointwise. By Lemma \ref{compact lemma section 2}, it follows that $\|\psi g_n\|_\mathcal{L} \to 0$ as $n \to \infty$. Moreover $$\begin{aligned}
\|\psi g_n\|_\mathcal{L} \geq |\psi(v_n)g_n(v_n) - \psi(v_n^-)g_n(v_n^-)|
= D\psi(v_n)\log|v_n|.
\end{aligned}$$  Therefore $\lim\limits_{n \to \infty}D\psi(v_n)\log|v_n| = 0.$

To prove the implication $(c)\Longrightarrow (a)$, suppose $\displaystyle\lim_{|v|\to\infty} \frac{|\psi(v)|}{|v|+1} = 0$ and $\displaystyle\lim_{|v|\to\infty} D\psi(v)\log|v| = 0$.  Clearly, if $\psi$ is identically 0, then $M_\psi$ is compact. So assume $M_\psi:\mathcal{L}_{\textbf{w}} \to \mathcal{L}$ is bounded with $\psi$ not identically 0. By Lemma~\ref{compact lemma section 2}, it suffices to show that if $\{f_n\}$ is bounded in $\mathcal{L}_{\textbf{w}}$ converging to 0 pointwise, then $\|\psi f_n\|_\mathcal{L} \to 0$ as $n \to \infty$.  Let $\{f_n\}$ be such a sequence, let $s = \displaystyle\sup_{n \in \N} \|f_n\|_{\textbf{w}}$, and fix $\e > 0$. Note that
$$\lim_{|v|\to\infty}D\psi(v)\log(1+|v|)=\lim_{|v|\to\infty}D\psi(v)\log|v|\frac{\log(1+|v|)}{\log|v|}=0.$$ Thus there exists an $M \in \N$ such that $$|f_n(o)|<\frac{\e}{3s\|\psi\|_\Lip},\ D\psi(v)\log(1+|v|) < \frac{\e}{6s} \text{ and } \frac{|\psi(v)|}{|v|+1} < \frac{\e}{3s},$$ for $|v| \geq M$. Using Lemma~\ref{variantnew}, for $|v|> M$, we have
$$\begin{aligned}
D(\psi f_n)(v) 
&\leq D\psi(v)|f_n(v)| + Df_n(v^-)|\psi(v^-)|\\
&\leq D\psi(v)\left(|f_n(o)|+2\log(|v|+1)\right)\|f_n\|_{\textbf{w}} + \|f_n\|_{\textbf{w}}\frac{|\psi(v^-)|}{|v|}\\
&\leq \left(\|\psi\|_\Lip|f_n(o)|+2D\psi(v)\log(|v|+1) + \frac{|\psi(v^-)|}{|v|}\right)\|f_n\|_{\textbf{w}}\\
&<\e.
\end{aligned}$$

On the other hand, on the set $B_M = \{v \in T : |v| \le M\}$, $\{f_n\}$ converges to 0 uniformly, and thus $Df_n$ does as well. Moreover
$$\begin{aligned}
D(\psi f_n)(v) &\le D\psi(v)|f_n(v)|+|\psi(v^-)|Df_n(v)\\
&\le \|\psi\|_\Lip|f_n(v)|+\max_{|w|\le M}|\psi(w)|Df_n(v)\to 0,\end{aligned}$$ uniformly on $B_M$. Therefore $D(\psi f_n)\to 0$ uniformly on $T$. Furthermore, the sequence $\{(\psi f_n)(o)\}$ converges to 0 as $n\to\infty$. Hence
 $\|\psi f_n\|_\mathcal{L} \to 0$ as $n \to \infty$, proving that $M_\psi$ is compact.

Finally, note that the functions $f_n$ and $g_n$ defined in the proof of $(a) \Longrightarrow (c)$ are in $\mathcal{L}_{w,0}$. So the equivalence of $(b)$ and $(c)$ is proved analogously.
\end{proof}
Recall the \textit{essential norm} of a bounded operator $S$ between Banach spaces $\mathcal{X}$ and $\mathcal{Y}$ is defined as $$\|S\|_e = \inf\;\{\|S-K\| : K \text{ is compact from $\mathcal{X}$ to $\mathcal{Y}$}\}.$$  For $\psi$ a function on $T$, define the quantities
$$\begin{aligned}
A(\psi) &= \lim_{n \to \infty} \sup_{|v|\geq n} \frac{|\psi(v)|}{|v|+1},\\
B(\psi) &= \lim_{n \to \infty} \sup_{|v|\geq n} D\psi(v)\log|v|.
\end{aligned}$$

\begin{theorem} Let $M_\psi$ be a bounded multiplication operator from {\rm $\mathcal{L}_{\textbf{w}}$} to $\mathcal{L}$.  Then $$\|M_\psi\|_e \geq \max\left\{A(\psi),B(\psi)\right\}.$$\end{theorem}

\begin{proof}
For each $n \in \N$, define $f_n = \frac{1}{n+1}\chi_n$, where $\chi_n$ denotes the characteristic function of the set $\{v\in T:|v|=n\}$.  Then $f_n \in \mathcal{L}_{{\textbf{w}},0}$, $\|f_n\|_{\textbf{w}} = 1$, and $f_n \to 0$ pointwise.  Thus, by Lemma~\ref{weakconv_Lipw}, $\{f_n\}$ converges to 0 weakly in $\mathcal{L}_{{\textbf{w}},0}$.  Let $\mathcal{K}$ be the set of compact operators from from $\mathcal{L}_{{\textbf{w}},0}$ to $\mathcal{L}_0$, and let $K \in \mathcal{K}$.  Then $K$ is completely continuous \cite{Conway:07}, and so $\|Kf_n\|_{\mathcal{L}} \to 0$ as $n \to \infty$.  Thus
$$\|M_\psi - K\| \geq \limsup_{n \to \infty} \|(M_\psi - K)f_n\|_{\mathcal{L}} \geq \limsup_{n \to \infty} \|M_\psi f_n\|_{\mathcal{L}}.$$  Now note that
$$\|M_\psi f_n\|_\Lip=\sup_{|v|=n}\frac{|\psi(v)|}{n+1}.$$
Hence
$$\begin{aligned}\|M_\psi\|_e &\geq \inf\{\|M_\psi - K\| : K \in \mathcal{K}\}\\
&\geq \limsup_{n\to\infty} \|M_\psi f_n\|_\mathcal{L}\\
&= \lim_{n \to \infty} \sup_{|v|\geq n} \frac{|\psi(v)|}{|v|+1}\\
&= A(\psi).
\end{aligned}$$  We will now show that $\|M_\psi\|_e \geq B(\psi)$.  This estimate is clearly true if $B(\psi) = 0$.  So assume $\{v_n\}$ is a sequence in $T$ such that $2 \leq |v_n| \to \infty$ as $n \to \infty$ and $$\lim_{n \to \infty} D\psi(v_n)\log|v_n| = B(\psi).$$  For $n \in \N$ and $v \in T$, define $$h_n(v) = \begin{cases}
\displaystyle\frac{\left[\log(|v|+1)\right]^2}{\log |v_n|} &\text{ if } 0 \leq |v| < |v_n|,\\
\log|v_n| &\text{ if } |v| \geq |v_n|.
\end{cases}$$  Then $h_n(o) = 0$, $h_n(v_n) = h_n(v_n^-) = \log|v_n|$, and $$|v|Dh_n(v) = \begin{cases}
\frac{|v|}{\log|v_n|}\log\left(\frac{|v|+1}{|v|}\right)\log\left[|v|(|v|+1)\right] &\text{ if } 1 \leq |v| < |v_n|,\\
\ \ \ 0 &\text{ if } |v| \geq |v_n|.
\end{cases}$$

The supremum of $|v|Dh_n(v)$ is attained at the vertices of length $|v_n|-1$ and is given by
$$s_n = \sup_{v \in T^*} |v|Dh_n(v) = (|v_n|-1)\log\left(\frac{|v_n|}{|v_n|-1}\right)\frac{\log\left[(|v_n|-1)|v_n|\right]}{\log|v_n|}.$$  Since $(|v_n|-1)\log\left(\frac{|v_n|}{|v_n|-1}\right) \leq 1$, we have
$$\frac{(\log 2)^2}{\log|v_n|} \leq \|h_n\|_{\textbf{w}} = s_n \leq \frac{\log\left[(|v_n|-1)|v_n|\right]}{\log|v_n|}<2.$$

By letting $g_n = \frac{h_n}{\|h_n\|_{\textbf{w}}}$, we have $g_n \in \mathcal{L}_{{\textbf{w}},0}$, $\|g_n\|_{\textbf{w}} = 1$, and $g_n \to 0$ pointwise.  By Lemma~\ref{weakconv_Lipw}, the sequence $\{g_n\}$ converges to 0 weakly in $\mathcal{L}_{{\textbf{w}},0}$.  Thus $\|Kg_n\|_\mathcal{L} \to 0$ as $n \to \infty$.  Therefore
$$\|M_\psi - K\| \geq \limsup_{n \to \infty} \|(M_\psi - K)g_n\|_\mathcal{L} \geq \limsup_{n \to \infty} \|\psi g_n\|_\mathcal{L}.$$

For each $n \in \N$, we have $g_n(v_n) = g_n(v_n^-) = \frac{\log|v_n|}{s_n}$.  So $$D(\psi g_n)(v_n) = \frac{1}{s_n}D\psi(v_n)\log|v_n|.$$  Since $\displaystyle\lim_{n \to \infty} s_n = 1$, we have
$$\begin{aligned}\|M_\psi\|_e &\geq \inf\{\|M_\psi - K\| : K \in \mathcal{K}\}\\
&\geq \limsup_{n \to \infty} \sup_{v \in T^*} D(\psi g_n)(v)\\
&\geq \lim_{n \to \infty}\frac{1}{s_n}D\psi(v_n)\log|v_n|\\
&= B(\psi).
\end{aligned}$$
Therefore, $\|M_\psi\|_e \geq \max\left\{A(\psi),B(\psi)\right\}$.
\end{proof}

Now now derive an upper estimate on the essential norm.

\begin{theorem}\label{upp_est_ess} Let $M_\psi$ be a bounded multiplication operator from $\mathcal{L}_{\textbf{w}}$ to $\mathcal{L}$.  Then $$\|M_\psi\|_e \leq A(\psi) + B(\psi).$$\end{theorem}

\begin{proof}
For $n \in \N$, define the operator $K_n$ on $\mathcal{L}_{\textbf{w}}$ by
$$(K_n f)(v) = \begin{cases}
f(v) &\text{ if } |v| \leq n,\\
f(v_n) &\text{ if } |v| > n,
\end{cases}$$ where $f \in \mathcal{L}_{\textbf{w}}$ and $v_n$ is the ancestor of $v$ of length $n$. For $f \in \mathcal{L}_{\textbf{w}}$, $(K_n f)(o) = f(o)$, and $K_n f \in \mathcal{L}_{{\textbf{w}},0}$. Let $B_n = \{v \in T : |v| \leq n\}$ and note that $K_n f$ attains finitely many values, whose number does not exceed the cardinality of $B_n$.  Let $\{g_k\}$ be a sequence in $\mathcal{L}_{\textbf{w}}$ such that $\|g_k\|_{\textbf{w}} \leq 1$ for each $k \in \N$.  Then $a = \displaystyle\sup_{k \in \N} |g_k(o)| \leq 1$ and $|K_ng_k(o)| \leq a$.  Furthermore, by part (a) of Lemma~\ref{new}, for each $v \in T^*$ and for each $k \in \N$, we have $|K_ng_k(v)| \leq 1 + \log n$.  Thus, some subsequence of $\{K_ng_k\}_{k \in \N}$ must converge to a function $g$ on $T$ attaining constant values on the sectors determined by the vertices of length $n$.  It follows that this subsequence converges to $g$ in $\mathcal{L}_{\textbf{w}}$ as well, proving that $K_n$ is a compact operator on $\mathcal{L}_{\textbf{w}}$. Since $M_\psi$ is bounded as an operator from $\mathcal{L}_{\textbf{w}}$ to $\mathcal{L}$, it follows that $M_\psi K_n:\mathcal{L}_{\textbf{w}}\to \mathcal{L}$ is compact for all $n \in \N$.

	Define the operator $J_n = I-K_n$, where $I$ denotes the identity operator on $\mathcal{L}_{\textbf{w}}$. Then $J_nf(o)=0$ and for $v \in T^*$, we have
\ben
|v|D(J_n f)(v) = |v||(J_n f)(v) - (J_n f)(v^-)|\leq |v|Df(v)\leq \|f\|_{\textbf{w}}.\label{equa1}\eeqn
 By part (a) of Lemma~\ref{new}, we see that \ben |(J_n f)(v)| \leq (1+\log|v|)\|f\|_{\textbf{w}}.\label{equa2}\eeqn

Using (\ref{equa1}) and (\ref{equa2}), we obtain
$$\begin{aligned}
\|(M_\psi - M_\psi K_n)f\|_\mathcal{L} &= \|\psi(J_n f)\|_\mathcal{L}\\
&= \sup_{|v| > n} |\psi(v)(J_n f)(v) - \psi(v^-)(J_n f)(v^-)|\\
&\leq \sup_{|v| > n} \left[|(J_n f)(v)|D\psi(v)+|\psi(v^-)|D(J_n f)(v)\right]\\
&= \sup_{|v| > n} \left[|(J_n f)(v)|D\psi(v) + \frac{|\psi(v^-)|}{|v|}|v|D(J_n f)(v)\right]\\
&\le \sup_{|v| \geq n} \left[(1+\log|v|)D\psi(v)+ \frac{|\psi(v)|}{|v|+1}\right]\|f\|_{\textbf{w}}\\
&\le \sup_{|v| \geq n} \left[\log|v|D\psi(v)\frac{1+\log|v|}{\log|v|}+ \frac{|\psi(v)|}{|v|+1}\right]\|f\|_{\textbf{w}}\\
&\le \left[\sup_{|v| \geq n} \log|v|D\psi(v)\frac{1+\log n}{\log n}
+ \sup_{|v| \geq n} \frac{|\psi(v)|}{|v|+1}\right]\|f\|_{\textbf{w}}.
\end{aligned}$$  Since $$\|M_\psi\|_e \leq \limsup_{n \to \infty} \|M_\psi - M_\psi K_n\| = \limsup_{n \to \infty} \sup_{\|f\|_{\textbf{w}} = 1} \|(M_\psi - M_\psi K_n)f\|_\mathcal{L},$$ taking the limit as $n \to \infty$, we obtain $$\|M_\psi\|_e \leq B(\psi)+A(\psi).\;\qedhere$$
\end{proof}

\section{Multiplication operators from $\Lip$ to $\Lip_{\textbf{w}}$}

We begin this section with a boundedness criterion for the multiplication operators from $M_\psi : \Lip \to \Lip_{\textbf{w}}$ and $M_\psi : \Lip_0 \to \Lip_{\textbf{w},0}$.  

\subsection{Boundedness and Operator Norm Estimates}
Let $\psi$ be a function on the tree $T$.  Define the quantities
$$\begin{aligned}
\theta_\psi &= \sup_{v\in T^*}|v|^2 D\psi(v),\\
\o_\psi &= \sup_{v\in T}(|v|+1)|\psi(v)|.
\end{aligned}$$

\begin{theorem}\label{liptolipm} For a function $\psi$ on $T$, the following statements are equivalent:
\begin{enumerate}
\item[\rm{(a)}] {\rm $M_\psi:\Lip\to \Lip_{\textbf{w}}$} is bounded.
\item[\rm{(b)}] {\rm $M_\psi:\Lip_0\to \Lip_{{\textbf{w}},0}$} is bounded.
\item[\rm{(c)}] $\theta_\psi$ and $\o_\psi$ are finite.
\end{enumerate}
Furthermore, under the above conditions, we have
$$\max\{\theta_\psi,\o_\psi\}\le \|M_\psi\|\le \theta_\psi+\o_\psi.$$
\end{theorem}

\begin{proof} $(a)\Longrightarrow (c)$ Assume $M_\psi$ is bounded from $\Lip$ to $\Lip_{\textbf{w}}$. The function $f_o=\frac12\chi_o\in \Lip$ and $\left\|f_o\right\|_\Lip=1$. Thus
\ben |\psi(o)|=\|\psi f_o\|_{\textbf{w}}\le \|M_\psi\|.\label{psi0}\eeqn

Next, fix $v\in T^*$. Then $\chi_v\in\Lip$ and $\|\chi_v\|_\Lip=1$; so
\ben (|v|+1)|\psi(v)|=\|\psi \chi_v\|_{\textbf{w}}\le \|M_\psi\|.\label{T*}\eeqn
Taking the supremum over all $v\in T$, from (\ref{psi0}) and (\ref{T*}) we see that $\o_\psi$ is finite and
\ben \o_\psi\le \|M_\psi\|.\label{low1}
\eeqn

With $v\in T^*$, we now define
$$f_v(w)=\begin{cases} |w| &\hbox{ if }|w|< |v|,\\
 |v| &\hbox{ if }|w|\ge |v|.\end{cases}$$
Then $f_v\in\Lip$, $f_v(o)=0$ and $\|f_v\|_\Lip=1$. By the boundedness of $M_\psi$ we obtain
\ben \|M_\psi\|&\ge& \|M_\psi f_v\|_{\textbf{w}}\ge \sup_{1\le |w|\le |v|}|w||\psi(w)|w|-\psi(w^-)(|w|-1)|\nonumber\\
&\ge &\sup_{1\le |w|\le |v|}|w|^2D\psi(w)-\sup_{1\le |w|\le |v|}|w||\psi(w^-)|,\nonumber\eeqn
Therefore $$|v|^2D\psi(v)\le \sup_{1\le|w|\le |v|}|w|^2D\psi(w)\le \|M_\psi\|+\o_\psi.$$
Taking the supremum over all $v\in T^*$, we obtain $\theta_\psi<\infty.$ From this and (\ref{low1}), we deduce the lower estimate $$\|M_\psi\|\ge \max\{\theta_\psi,\o_\psi\}.$$

$(c)\Longrightarrow (a)$ Assume $\theta_\psi$ and $\o_\psi$ are finite. Then, $\psi\in\Lip_{\textbf{w}}$ and by Lemma~\ref{old}, for $f\in\Lip$ with $\|f\|_\Lip=1$ and $v\in T^*$, we have
\ben |v|D(\psi f)(v)&\le& |v|D\psi (v)|f(v)|+|v||\psi(v^-)|Df(v)\nonumber\\
&\le & |v |D\psi(v)|f(o)|+|v|^2 D\psi(v)\|Df\|_\infty+\o_\psi\|Df\|_\infty\nonumber\\
&\le & |v |D\psi(v)|f(o)|+(\theta_\psi+\o_\psi)\|Df\|_\infty.
\nonumber\eeqn
Thus, $\psi f\in\Lip_{\textbf{w}}$.  Note that $|f(o) + \|Df\|_\infty=1$ and
$$\|\psi\|_{\textbf{w}}=|\psi(o)|+\sup_{v\in T^*}|v|D\psi(v)\le \o_\psi+\sup_{v\in T^*}|v|^2 D\psi(v)=\o_\psi+\theta_\psi.$$  From this, we have
$$\|\psi f\|_{\textbf{w}}\le \|\psi\|_{\textbf{w}}|f(o)|+(\theta_\psi+\o_\psi)\|Df\|_\infty\le \theta_\psi+\o_\psi,$$
proving the boundedness of $M_\psi:\Lip\to \Lip_{\textbf{w}}$ and the upper estimate
$$\|M_\psi\|\le \theta_\psi+\o_\psi.$$

$(b)\Longrightarrow (c)$ The proof is the same as for $(a)\Longrightarrow (c)$, since for $v\in T^*$, the functions $\chi_v$ and $f_v$ used there belong to $\Lip_0$.

$(c)\Longrightarrow (b)$ Assume $\theta_\psi$ and $\omega_\psi$ are finite and let $f\in\Lip_{0}$. Then, by Lemma~\ref{old}, for $v\in T^*$, we have
\ben |v|D(\psi f)(v)&\le &|v|D\psi(v)|f(v)|+|v||\psi(v^-)|Df(v)\nonumber\\
&\le &|v|^2D\psi(v)\frac{|f(v)|}{|v|}+|v||\psi(v^-)|Df(v)\nonumber\\
&\le&\theta_\psi\frac{|f(v)|}{|v|}+\o_\psi Df(v)\to 0\nonumber\eeqn
as $|v|\to\infty$. Thus, $\psi f\in\Lip_{{\textbf{w}},0}.$ The proof of the boundedness of $M_\psi$ is similar to that in $(c)\Longrightarrow (a)$. 
\end{proof}

\subsection{Isometries}
In this section, we show there are no isometric multiplication operators $M_\psi$ from the space $\Lip$ to $\Lip_{\textbf{w}}$ or from $\Lip_0$ to $\Lip_{\textbf{w},0}$.

Suppose $M_\psi:\Lip\to\Lip_{\textbf{w}}$ is an isometry. Then
$\|\psi\|_{\textbf{w}}=\|M_\psi 1\|_{\textbf{w}}=1.$ On the other hand, $$|\psi(o)|=\frac12\left\|\psi\chi_o\right\|_{\textbf{w}}=\frac12\left\|\chi_o\right\|_\Lip=1.$$
Thus $\sup\limits_{v\in T^*}|v|D\psi(v)=\|\psi\|_{\textbf{w}}-|\psi(o)|=0$, which implies that $\psi$ is a constant of modulus 1. Now observe that for $v\in T^*$, we have
$$1=\|\chi_v\|_\Lip=\|M_\psi\chi\|_{\textbf{w}}=(|v|+1)|\psi(v)|=|v|+1,$$
which is a contradiction. Since $\chi_v\in\Lip_0$ for all $v \in T$, if $M_\psi:\Lip_0\to\Lip_{\textbf{w},0}$ is an isometry, then the above argument yields again a contradiction. Thus, we proved the following result.

\begin{theorem}\label{noiso3} There are no isometries $M_\psi$ from $\Lip$ to {\rm $\Lip_{\textbf{w}}$} or from $\Lip_0$ to {\rm $\Lip_{{\textbf{w}},0}$}.
\end{theorem}

\subsection{Compactness and Essential Norm}
We now characterize the compact multiplication operators, but first we first give a useful compactness criterion for multiplication operators from $\Lip$ to $\Lip_{\textbf{w}}$ or from $\Lip_0$ to $\Lip_{\textbf{w},0}$.

\begin{lemma}\label{compact lemma section 3} A bounded multiplication operator $M_\psi$ from $\mathcal{L}$ to {\rm $\mathcal{L}_{\textbf{w}}$} (or from $\mathcal{L}_0$ to {\rm $\mathcal{L}_{\textbf{w},0}$}) is compact if and only if for every bounded sequence $\{f_n\}$ in $\mathcal{L}$ (respectively, $\mathcal{L}_0$) converging to 0 pointwise, the sequence {\rm $\|\psi f_n\|_{\textbf{w}}$} converges to 0 as $n \to \infty$.\end{lemma}

\begin{proof}
Suppose $M_\psi$ is compact from $\mathcal{L}$ to $\mathcal{L}_{\textbf{w}}$ and $\{f_n\}$ is a bounded sequence in $\mathcal{L}$ converging to 0 pointwise.  Without loss of generality, we may assume $\|f_n\|_\mathcal{L} \leq 1$ for all $n \in \N$.  Since $M_\psi$ is compact, the sequence $\{\psi f_n\}$ has a subsequence $\{\psi f_{n_k}\}$ that converges in the $\mathcal{L}_{\textbf{w}}$-norm to some function $f \in \mathcal{L}_{\textbf{w}}$.

By Lemma~\ref{new}, for $v \in T^*$ we have
$$|(\psi(v)f_{n_k}(v) - f(v)| \leq (1+\log |v|)\|\psi f_{n_k} - f\|_{\textbf{w}}.$$  Thus, $\psi f_{n_k} \to f$ pointwise on $T^*$.  Furthermore, since $|(\psi(o)f_{n_k}(o) - f(o)| \leq \|\psi f_{n_k} - f\|_{\textbf{w}}$, $\psi(0)f_{n_k}(0) \to f(0)$ as $k\to\infty$. Thus $\psi f_{n_k} \to f$ pointwise on $T$. Since by assumption, $f_n \to 0$ pointwise, it follows that $f$ is identically 0, and thus $\|\psi f_{n_k}\|_{\textbf{w}} \to 0$.  Since 0 is the only limit point in $\mathcal{L}_{\textbf{w}}$ of the sequence $\{\psi f_n\}$, we deduce that $\|\psi f_n\|_{\textbf{w}} \to 0$ as $n \to \infty$.

Conversely, suppose that every bounded sequence $\{f_n\}$ in $\mathcal{L}$ that converges to 0 pointwise has the property that $\|\psi f_n\|_{\textbf{w}} \to 0$ as $n \to \infty$.  Let $\{g_n\}$ be a sequence in $\mathcal{L}$ such that $\|g_n\|_\mathcal{L} \leq 1$ for all $n \in \N$.  Then $|g_n(o)| \leq 1$, and by part (a) of Lemma \ref{old}, for $v \in T^*$ we have $|g_n(v)| \leq |v|$.  So $\{g_n\}$ is uniformly bounded on finite subsets of $T$.  Thus there is a subsequence $\{g_{n_k}\}$, which converges pointwise to some function $g$.

Fix $\varepsilon > 0$ and $v \in T^*$.  Then $|g_{n_k}(v) - g(v)| < \frac{\varepsilon}{2}$ as well as $|g_{n_k}(v^-)-g(v^-)| < \frac{\varepsilon}{2}$ for $k$ sufficiently large.  Therefore, for all $k$ sufficiently large, we have
$$Dg(v) \leq |g(v)-g_{n_k}(v)| + |g_{n_k}(v^-) - g(v^-)| + Dg_{n_k}(v) < \varepsilon + Dg_{n_k}(v).$$  Thus $g \in \mathcal{L}$.  The sequence $f_{n_k} = g_{n_k} - g$ is bounded in $\mathcal{L}$ and converges to 0 pointwise. So $\|\psi f_{n_k}\|_{\textbf{w}} \to 0$ as $k \to \infty$.  Thus $\psi g_{n_k} \to \psi g$ in the $\mathcal{L}_{\textbf{w}}$-norm. Therefore, $M_\psi$ is compact.

The proof for the case of $M_\psi:\mathcal{L}_0\to \mathcal{L}_{\textbf{w},0}$ is similar.
\end{proof}

\begin{theorem}\label{compact-section3} Let $M_\psi$ be a bounded multiplication operator from $\mathcal{L}$ to {\rm $\mathcal{L}_{\textbf{w}}$} (or equivalently from $\mathcal{L}_0$ to {\rm $\mathcal{L}_{{\textbf{w}},0}$}).  Then the following are equivalent:
\begin{enumerate}
\item[(a)] {\rm $M_\psi:\mathcal{L} \to \mathcal{L}_{\textbf{w}}$} is compact.
\item[(b)] {\rm $M_\psi:\mathcal{L}_0 \to \mathcal{L}_{{\textbf{w}},0}$} is compact.
\item[(c)] $\displaystyle\lim_{|v| \to \infty} |v|^2D\psi(v) = 0$ and $\displaystyle\lim_{|v| \to \infty} (|v|+1)|\psi(v)| = 0$.
\end{enumerate}\end{theorem}

\begin{proof}
$(a) \Longrightarrow (c)$  Suppose $M_\psi:\mathcal{L}\to \mathcal{L}_{\textbf{w}}$ is compact.  We need to show that if $\{v_n\}$ is a sequence in $T$ such that $2 \leq |v_n|$ increasing unboundedly, then $\displaystyle\lim_{n \to \infty} |v_n|^2D\psi(v_n) = 0$ and $\displaystyle\lim_{n \to \infty}(|v_n|+1)|\psi(v_n)| = 0$.  Let $\{v_n\}$ be such a sequence, and for $n \in \N$ define $f_n = \frac{|v_n|+1}{|v_n|}\chi_{v_n}$.  Clearly $f_n \to 0$ pointwise, and $\|f_n\|_{\mathcal{L}} \leq \frac{3}{2}$. Using Lemma~\ref{compact lemma section 3}, we see that
\ben \|\psi f_n\|_{\textbf{w}}\to 0 \ \hbox{ as } n\to\infty.\label{normto0}\eeqn

On the other hand, since $f_n(o) = 0$ for all $n \in \N$, we have
$$\|\psi f_n\|_{\textbf{w}} = \sup_{v \in T^*} |v|D(\psi f_n)(v) \geq |v_n|\left(\frac{|v_n| + 1}{|v_n|}\right)|\psi(v_n)| = (|v_n|+1)|\psi(v_n)|.$$  Hence $\displaystyle\lim_{n \to \infty} (|v_n|+1)|\psi(v_n)| = 0$.

Next, for $n \in \N$, define
$$g_n(v) = \begin{cases}
0 &\text{ if } |v| < \left\lfloor\frac{|v_n|}{2}\right\rfloor,\\
2|v| - |v_n| + 2 &\text{ if } \left\lfloor\frac{|v_n|}{2}\right\rfloor \leq |v| < |v_n|,\\
|v_n| &\text{ if } |v| \geq |v_n|.
\end{cases}$$  Then $g_n \to 0$ pointwise, and $\|g_n\|_\mathcal{L} = 2$.  Since $g_n(v_n) = g_n(v_n^-) = |v_n|$, we have $$
\|\psi g_n\|_w \geq |v_n||\psi(v_n)g_n(v_n) - \psi(v_n^-)g_n(v_n^-)| = |v_n|^2D\psi(v_n).$$  By Lemma \ref{compact lemma section 3} we obtain $\displaystyle\lim_{n \to \infty} |v_n|^2D\psi(v_n) \leq \lim_{n\to\infty} \|\psi g_n\|_{\text{w}} = 0$.

$(c) \Longrightarrow (a)$  Suppose $\displaystyle\lim_{|v|\to\infty} |v|^2D\psi(v) = 0$ and $\displaystyle\lim_{|v|\to\infty} (|v|+1)|\psi(v)| = 0$.  Assume $\psi$ is not identically zero, otherwise $M_\psi$ is trivially compact.  By Lemma~\ref{compact lemma section 3}, to prove that $M_\psi$ is compact, it suffices to show that if $\{f_n\}$ is a bounded sequence in $\mathcal{L}$ converging to 0 pointwise, then $\|\psi f_n\|_{\textbf{w}} \to 0$ as $n \to \infty$.  Let $\{f_n\}$ be such a bounded sequence, let $s = \displaystyle\sup_{v \in T} \|f_n\|_\mathcal{L}$, and fix $\e > 0$.  There exists $M \in \N$ such that $(|v|+1)|\psi(v)| < \frac{\e}{2s}$ and $|v|^2D\psi(v) < \frac{\e}{2s}$ for $|v| \geq M$.  For $v \in T^*$ and by Lemma \ref{old}, we have
$$\begin{aligned}
|v|D(\psi f_n)(v) &\leq |v||\psi(v)|Df_n(v) + |v|D\psi(v)|f_n(v^-)|\\
&\leq |v||\psi(v)|Df_n(v) + |v|D\psi(v)(|f_n(o)| + |v|\|Df_n\|_\infty)\\
&\leq (|v|+1)|\psi(v)|Df_n(v) + |v|^2D\psi(v)(|f_n(o)| + \|Df_n\|_\infty)\\
&=(|v|+1)|\psi(v)|Df_n(v) + |v|^2D\psi(v)\|f_n\|_\Lip.
\end{aligned}$$  Since $f_n \to 0$ uniformly on $\{v \in T : |v| \leq M\}$ as $n \to \infty$, so does $Df_n$.  So, on the set $\{v \in T : |v| \leq M\}$, $|v|D(\psi f_n)(v) \to 0$ as $n \to \infty$.  On the other hand, on $\{v \in T : |v| \geq M\}$, we have $$|v|D(\psi f_n)(v) \leq (|v|+1)|\psi(v)|Df_n(v) + |v|^2D\psi(v)\|f_n\|_\Lip < \varepsilon.$$  So $|v|D(\psi f_n)(v) \to 0$ as $n \to \infty$.  Since $f_n \to 0$ pointwise, $\psi(o)f_n(o) \to 0$ as $n\to\infty$.  Thus $\|\psi f_n\|_{\textbf{w}} \to 0$ as $n \to \infty$. The compactness of $M_\psi$ follows at once from Lemma~\ref{compact lemma section 3}.

The proof of the equivalence of $(b)$ and $(c)$ is analogous.
\end{proof}

For $\psi$ a function on $T$, define $$\begin{aligned}
\mathcal{A}(\psi) &= \lim_{n \to \infty} \sup_{|v|\geq n} |v||\psi(v)|,\\
\mathcal{B}(\psi) &= \lim_{n \to \infty} \sup_{|v|\geq n} |v|^2D\psi(v).
\end{aligned}$$

\begin{theorem} Let $M_\psi$ be a bounded multiplication operator from $\mathcal{L}$ to {\rm $\mathcal{L}_{\textbf{w}}$}.  Then $$\|M_\psi\|_e \geq \max\left\{\mathcal{A}(\psi), \frac{1}{2}\mathcal{B}(\psi)\right\}.$$\end{theorem}

\begin{proof} Fix $k\in \N$ and for each $n\in \N$, consider the sets
\ben E_{n,k}&=&\{v\in T: n\le |v|\le kn, |v| \hbox{ even}\},\nonumber\\
O_{n,k}&=&\{v\in T: n\le |v|\le kn, |v| \hbox{ odd}\}.\nonumber\eeqn
Define the functions $f_{n,k} = \chi_{E_{n,k}}$ and $g_{n,k} = \chi_{O_{n,k}}$.  Then $f_{n,k},g_{n,k} \in \mathcal{L}_0$, $\|f_{n,k}\|_{\mathcal{L}}=\|g_{n,k}\|_{\mathcal{L}} = 1$, and $f_n$ and $g_{n,k} \to 0$ pointwise as $n\to\infty$. By Lemma~\ref{weakconv_Lip}, the sequences $f_{n,k}$ and $g_{n,k}$ approach 0 weakly in $\mathcal{L}_0$ as $n\to\infty$.  Let $\mathcal{K}_0$ be the set of compact operators from $\Lip_0$ to {\rm $\Lip_{{\textbf{w}},0}$}, and note that every operator in $\mathcal{K}_0$ is completely continuous. Thus, if $K \in \mathcal{K}_0$, then $\|Kf_{n,k}\|_{\textbf{w}} \to 0$ and $\|Kg_{n,k}\|_{\textbf{w}} \to 0$, as $n\to\infty$.

Therefore, if $K \in \mathcal{K}_0$, then
\ben \|M_\psi - K\| &\geq& \limsup_{n \to\infty}\|(M_\psi - K)f_{n,k}\|_{\textbf{w}} \nonumber\\&\geq &\limsup_{n \to\infty}
 \|M_\psi f_{n,k}\|_{\textbf{w}}\nonumber\\
&\ge &\limsup_{n \to\infty}\sup_{v\in E_{n,k}}(|v|+1)|\psi(v)|.\label{even}\eeqn

Similarly, \ben \|M_\psi - K\| \geq \limsup_{n \to\infty}\sup_{v\in O_{n,k}}(|v|+1)|\psi(v)|.\label{odd}\eeqn
Therefore, combining (\ref{even}) and (\ref{odd}), we obtain
$$\begin{aligned}
\|M_\psi\|_e &= \inf\{\|M_\psi - K\| : K \in \mathcal{K}_0\}\\
&\ge \limsup_{n \to \infty} \sup_{kn\ge |v|\geq n} (|v|+1)|\psi(v)|\\
&\ge \limsup_{n \to \infty}\sup_{kn\ge |v|\geq n} |v||\psi(v)|.\end{aligned}$$
Letting $k\to\infty$, we obtain $\|M_\psi\|_e\ge \mathcal{A}(\psi).$

Next, we wish to show that $\|M_\psi\|_e \geq \frac12\mathcal{B}(\psi)$.  The result is clearly true if $\mathcal{B}(\psi) = 0$.  So assume there exists a sequence $\{v_n\}$ in $T$ such that $2 < |v_n| \to \infty$ as $n \to \infty$ and $$\lim_{n \to \infty} |v_n|^2D\psi(v_n) = \mathcal{B}(\psi).$$  For $n \in \N$, define $$h_n(v) = \begin{cases}
0 &\text{ if } v = o,\\
\frac{(|v|+1)^2}{|v_n|} &\text{ if } 1 \leq |v| < |v_n|,\\
|v_n| &\text{ if } |v|\geq |v_n|.
\end{cases}$$  Clearly, $h_n(o) = 0, h_n(v_n) = h_n(v_n^-) = |v_n|$, and
$$Dh_n(v) = \begin{cases}
\frac{4}{|v_n|} &\text{ if } |v|=1,\\
\frac{2|v|+1}{|v_n|} &\text{ if } 1 < |v| < |v_n|,\\
0 &\text{ if } |v| \geq |v_n|.
\end{cases}$$
The supremum of $Dh_n(v)$ is attained on the set $\{v \in T : |v| = |v_n|-1\}$.  Thus $\|h_n\|_\mathcal{L} = \frac{2|v_n|-1}{|v_n|} < 2$.  Define $g_n = \frac{h_n}{\|h_n\|_{\mathcal{L}}}$, and observe that $g_n \in \mathcal{L}_0$, $\|g_n\|_{\mathcal{L}} = 1$, and $g_n \to 0$ pointwise on $T$.  By Lemma~\ref{weakconv_Lip}, $g_n \to 0$ weakly in $\mathcal{L}_0$.  Thus $\|Kg_n\|_{\textbf{w}} \to 0$ as $n \to \infty$ for any $K \in \mathcal{K}_0$.

For each $n \in \N$, $g_n(v_n) = g_n(v_n^-) = \frac{|v_n|^2}{2|v_n|-1}$.  Thus
$$\begin{aligned}
|v_n|D(\psi g_n)(v_n) &= |v_n||\psi(v_n)g_n(v_n) - \psi(v_n^-)g_n(v_n^-)|\\
&= \frac{|v_n|}{2|v_n|-1}|v_n|^2D\psi(v_n).
\end{aligned}$$

We deduce that
$$\begin{aligned}
\|M_\psi\|_e &= \inf\{\|M_\psi - K\| : K \in \mathcal{K}_0\}\\
&\geq \limsup_{n \to \infty} \|(M_\psi-K)g_n\|_{\textbf{w}}\\
&\geq \limsup_{n \to \infty} \|M_\psi g_n\|_{\textbf{w}}\\
&\geq \lim_{n \to \infty}\sup_{v \in T^*} |v|D(\psi g_n)(v)\\
&\geq \lim_{n \to \infty} |v_n|D(\psi g_n)(v_n)\\
&= \lim_{n \to \infty} \frac{|v_n|}{2|v_n|-1}|v_n|^2D\psi(v_n)\\
&\geq \frac{1}{2}\mathcal{B}(\psi).
\end{aligned}$$

Therefore, $$\|M_\psi\|_e \geq \max\left\{\mathcal{A}(\psi), \frac{1}{2}\mathcal{B}(\psi)\right\}.\;\qedhere$$
\end{proof}

	We next derive an upper estimate on the essential norm.

\begin{theorem}\label{upper_ess_est} Let $M_\psi$ be a bounded multiplication operator from $\mathcal{L}$ to $\mathcal{L}_{\textbf{w}}$.  Then $$\|M_\psi\|_e \leq \mathcal{A}(\psi) + \mathcal{B}(\psi).$$\end{theorem}

\begin{proof}  For each $n \in \N$, consider the operator $K_n$ defined by $$(K_nf)(v) = \begin{cases}
f(v) &\text{ if } |v| \leq n,\\
f(v_n) &\text{ if } |v| > n,
\end{cases}$$ for $f \in \mathcal{L}$, where $v_n$ is the ancestor of $v$ of length $n$.  Then $(K_n f)(o) = f(o)$, and $K_n f \in \mathcal{L}_{\textbf{w},0}$. Arguing as in the proof of Theorem~\ref{upp_est_ess}, by the boundedness of $M_\psi$, it follows that $M_\psi K_n$ is a compact operator from $\mathcal{L}$ to $\mathcal{L}_{\textbf{w}}$.

Define the operator $J_n = I-K_n$, where $I$ is  the identity operator $I$ on $\mathcal{L}$. Then,
$$D(J_n f)(v) \leq Df(v) \leq \|f\|_\mathcal{L}.$$  Since $(J_n f)(v) = 0$ for $|v|\le n$, by Lemma~\ref{old}, we obtain $$|(J_nf)(v)| \leq |v|\|f\|_\mathcal{L}.$$  From these two estimates, we arrive at
\ben
\|M_\psi J_n f\|_{\textbf{w}}&=& \sup_{|v|>n} |v|\left|\psi(v)(J_n f)(v) - \psi(v^-)(J_n f)(v^-)\right|\nonumber\\
&\leq &\sup_{|v|>n} \left[|v|D\psi(v)|(J_n f)(v)|+ |v||\psi(v^-)|D(J_n f)(v)\right]\nonumber\\
&\leq &\sup_{|v|>n} |v|^2D\psi(v)\frac{|(J_n f)(v)|}{|v|} + \sup_{|v|>n} |v||\psi(v^-)|D(J_n f)(v)|\nonumber\\
&\leq &\sup_{|v|>n} |v|^2D\psi(v)\|f\|_\mathcal{L} + \sup_{|v|>n} |v||\psi(v^-)|\|f\|_\mathcal{L}. \label{est_up}
\eeqn  Since $$\begin{aligned}
\|M_\psi\|_e &\leq \limsup_{n \to \infty} \|M_\psi - M_\psi K_n\|\\
&= \limsup_{n \to \infty} \sup_{\|f\|_\mathcal{L} = 1} \|(M_\psi - M_\psi K_n)f\|_{\textbf{w}}\\
&= \limsup_{n \to \infty} \sup_{\|f\|_\mathcal{L} = 1}\|M_\psi J_n f\|_{\textbf{w}},
\end{aligned}$$ from (\ref{est_up}), taking the limit as $n \to \infty$, we obtain
$$\|M_\psi\|_e \leq \mathcal{B}(\psi) + \mathcal{A}(\psi).\;\qedhere$$
\end{proof}

\section{Multiplication operators from $\Lip_{\textbf{w}}$ or $\Lip_{{\textbf{w}},0}$ to $L^\infty$}
In this section, we study the multiplication operators $M_\psi$ from the weighted Lipschitz space or the little weighted Lipschitz space into $L^\infty$.  We begin by characterizing the bounded operators and determining their operator norm. In addition, we characterize the bounded operators that are bounded from below and show that there are no isometries among them. Finally, we characterize the compact multiplication operators and determine the essential norm.

\subsection{Boundedness and Operator Norm}
For a function $\psi$ on $T$, define
$$\g_\psi=\max\left\{|\psi(o)|,\sup_{v\in T^*}(1+\log|v|)|\psi(v)|\right\}.$$

\begin{theorem}\label{boundedness} For a function $\psi$ on $T$, the following statements are equivalent:
\begin{enumerate}
\item[\rm{(a)}] {\rm $M_\psi:\Lip_{\textbf{w}}\to L^\infty$} is bounded.
\item[\rm{(b)}] {\rm $M_\psi:\Lip_{{\textbf{w}},0}\to L^\infty$} is bounded.
\item[\rm{(c)}] $\sup_{v\ne o}\log|v||\psi(v)|$ is finite.
\end{enumerate}
Furthermore, under the above conditions, we have $\|M_\psi\|=\g_\psi.$
\end{theorem}

\begin{proof} The implication $(a)\Longrightarrow (b)$ is obvious.

$(b)\Longrightarrow (a)$: We begin by showing that for each $f\in\Lip_{\textbf{w}}$, the function $\psi f$ is bounded. Since $M_\psi$ is bounded on $\Lip_{{\textbf{w}},0}$, $\psi=M_\psi 1\in L^\infty$. Thus, if $f$ is constant, then $\psi f\in L^\infty$. Fix $f\in\Lip_{\textbf{w}}$, $f$ nonconstant, $v\in T$, and set $n=|v|$. For $w\in T$, define
$$f_n(w)=\begin{cases} f(w)& \quad \hbox{ if }|w|\le n\\
f(w_n)& \quad \hbox{ if }|w|> n\end{cases}$$
where $w_n$ is the ancestor of $w$ of length $n$. Then $f_n\in\Lip_{{\textbf{w}},0}$ and $\|f_n\|_{\textbf{w}}\le \|f\|_{\textbf{w}}$. Thus, $\psi f_n\in L^\infty$ and
$$\|\psi f_n\|_\infty\le \|M_\psi\|\,\|f\|_{\textbf{w}}.$$
So $|\psi(v)f(v)|=|\psi(v)f_n(v)|\le \|M_\psi\|\,\|f\|_{\textbf{w}}.$ Therefore $\psi f\in L^\infty$ and
$$\|\psi f\|_\infty\le \|M_\psi\|\,\|f\|_{\textbf{w}},$$
proving the boundedness of $M_\psi$ as an operator from $\Lip_{\textbf{w}}$ to $L^\infty$.

$(a)\Longrightarrow (c)$: Assume $M_\psi:\Lip_{\textbf{w}}\to L^\infty$ is bounded. Then $\psi=M_\psi 1\in L^\infty$ and
\ben \|M_\psi\|\ge \|\psi\|_\infty\ge |\psi(o)|.\label{est1}\eeqn
 For $v\in T$, define $f(v)=\log(1+|v|)$.
Then $f(o)=0$ and since for $x\ge 1$ the function $x\mapsto x\log\left(\frac{x+1}{x}\right)$ is increasing and has limit 1 as $x\to\infty$, $f\in\Lip_{\textbf{w}}$ and $\|f\|_{\textbf{w}}=1.$ Thus
\ben \|M_\psi\|\ge \|\psi f\|_\infty=\sup_{v\in T^*}\log(1+|v|)|\psi(v)|,\label{est2}\eeqn proving (c). Furthermore, from (\ref{est1}) and (\ref{est2}), we obtain
\ben \|M_\psi\|\ge \g_\psi.\label{estlower}\eeqn

$(c)\Longrightarrow (a)$: Assume $\displaystyle\sup_{v\ne o}\log|v||\psi(v)|<\infty.$ Let $f\in\Lip_{\textbf{w}}$ such that $\|f\|_{\textbf{w}}=1$. Then
$|\psi(o)f(o)|\le |\psi(o)|$ and  by Lemma~\ref{new}, for $v\in T^*$, we have
$$|\psi(v)f(v)|\le (1+\log|v|)|\psi(v)|\le \g_\psi.$$
Thus, $\psi f\in L^\infty$ and
\ben \|\psi f\|_\infty \le \g_\psi,\label{est3}\eeqn proving the boundedness of $M_\psi$ as an operator from $\Lip_{\textbf{w}}$ to $L^\infty$.
Taking the supremum over all functions $f\in\Lip_{\textbf{w}}$ such that $\|f\|_{\textbf{w}}=1$, from (\ref{est3}) we obtain $\|M_\psi\|\le \g_\psi.$
Therefore, from (\ref{estlower}) we conclude that
$\|M_\psi\|=\g_\psi.$
\end{proof}

\subsection{Boundedness From Below}  Recall that an operator $S$ from a Banach space $\mathcal{X}$ to a Banach space $\mathcal{Y}$ is {\it{bounded below}} if there exists a constant $C>0$ such that for all $x\in X$ $$\|Sx\|\ge C\|x\|.$$

\begin{theorem}\label{bbelow} A bounded multiplication operator $M_\psi$ from {\rm $\Lip_{\textbf{w}}$} or {\rm $\Lip_{{\textbf{w}},0}$} to $L^\infty$ is bounded below if and only if $$\inf_{v\in T}\frac{|\psi(v)|}{|v|+1}>0.$$
\end{theorem}

\begin{proof} Assume $M_\psi$ is bounded below and, arguing by contradiction, assume there exists $v\in T$ such that $\psi(v)=0$. Then $M_\psi\chi_v$ is identically 0. Since operators that are bounded below are necessarily injective \cite{Conway:07}, it follows that $M_\psi$ is not bounded below. Therefore, if $M_\psi$ is bounded below, then $\psi$ is nonvanishing.

Next assume $\psi$ is nonvanishing and $\inf\limits_{v\in T}\frac{|\psi(v)|}{|v|+1}=0.$ Then, there exists a sequence $\{v_n\}$ in $T$ with $1\le |v_n|\to\infty$, such that $\frac{|\psi(v_n)|}{|v_n|+1}\to 0$ as $n\to\infty$. For $n\in\N$, define $f_n=\frac1{|v_n|+1}\chi_{v_n}$. Then $\|f_n\|_{\textbf{w}}=1$, but
$$\|\psi f_n\|_\infty=\frac{|\psi(v_n)|}{|v_n|+1}\to 0.$$
Thus, $M_\psi$ is not bounded below.

Conversely, assume $\inf\limits_{v\in T}\frac{|\psi(v)|}{|v|+1}=c>0$ and that $M_\psi$ is not bounded below. Then, for each $n\in\N$, there exists $f_n\in\;${\rm $\Lip_{\textbf{w}}$} such that $\|f_n\|_{\textbf{w}}=1$ and $\|\psi f_n\|_\infty<\frac{1}{n}$. Then, for each $v\in T$, we have
$$c(|v|+1)|f_n(v)|\le |\psi(v)f_n(v)|<\frac1n,$$ so that the sequence $\{g_n\}$ defined by $g_n(v)=(|v|+1)f_n(v)$ converges to 0 uniformly.

On the other hand, for $v\in T^*$, we have
\ben |v|Df_n(v)&=&\left|\frac{|v|}{|v|+1}g_n(v)-g_n(v^-)\right|\nonumber\\
&\le &|g_n(v)|+|g_n(v^-)|\to 0\nonumber\eeqn
uniformly as $n\to\infty$. Since $|\psi(o)f_n(o)|<1/n$, yet $\|f_n\|_{\textbf{w}}=1$, this yields a contradiction.
 \end{proof}

\subsection{Isometries}
In this section, we show there are no isometries among the multiplication operators from the spaces $\Lip_{\textbf{w}}$ or $\Lip_{\textbf{w},0}$ into $L^\infty$.

Suppose $M_\psi$ is an isometry from $\Lip_{\textbf{w}}$ or $\Lip_{{\textbf{w}},0}$ to $L^\infty$. Then, for $v\in T$ the function $f_v=\frac1{|v|+1}\chi_v$ is in $\Lip_{{\textbf{w}},0}$, $\|f_v\|_{\textbf{w}}=1$, and $$\frac1{|v|+1}|\psi(v)|=\|M_\psi f_v\|_\infty=\|f_v\|_{\textbf{w}}=1.$$ Thus, $|\psi(v)|=|v|+1$. On the other hand, since $M_\psi$ is bounded, by Theorem~\ref{boundedness}, we have $\lim\limits_{|v|\to\infty}|v|\log|v||\psi(v)=0$; so $\psi(v)\to 0$ as $|v|\to\infty$, which yields a contradiction.
Thus, we proved the following result.

\begin{theorem}\label{noiso} The are no isometric multiplication operators $M_\psi$ from {\rm $\Lip_{\textbf{w}}$} or {\rm $\Lip_{{\textbf{w}},0}$} to $L^\infty$.
\end{theorem}

\subsection{Compactness and Essential Norm}
We begin by giving a useful compactness criterion for the bounded operators from $\Lip_{\textbf{w}}$ or $\Lip_{\textbf{w},0}$ into $L^\infty$.

\begin{lemma}\label{compact:chara} A bounded multiplication operator $M_\psi$ from {\rm $\Lip_{\textbf{w}}$} to $L^\infty$ is compact if and only if for every bounded sequence $\{f_n\}$ in {\rm $\Lip_{\textbf{w}}$} converging to 0 pointwise, the sequence $\|\psi f_n\|_\infty$ approaches 0 as $n\to\infty$.\end{lemma}

\begin{proof}  Assume $M_\psi$ is compact on $\Lip_{\textbf{w}}$ and let $\{f_n\}$ be a bounded sequence in $\Lip_{\textbf{w}}$ converging to 0 pointwise. By rescaling the sequence, if necessary, we may assume $\|f_n\|_{\textbf{w}} \le 1$ for all $n\in \N$. By the compactness of $M_\psi$, $\{f_n\}$ has a subsequence $\{f_{n_k}\}$ such that $\{\psi f_{n_k}\}$ converges in the supremum-norm to some function $f\in L^\infty$.
In particular, $\psi f_{n_k}\to f$ pointwise. Since by assumption, $f_n\to 0$ pointwise, it follows that $f$ must be identically 0. Thus, the only limit point of the sequence $\{\psi f_n\}$ in $L^\infty$ is 0. Hence $\|\psi f_n\|_\infty\to 0$.

	Conversely, assume that for every bounded sequence $\{f_n\}$ in $\Lip_{\textbf{w}}$ converging to 0 pointwise, the sequence $\|\psi f_n\|_\infty$ approaches 0 as $n\to\infty$. Let $\{g_n\}$ be a sequence in $\Lip_{\textbf{w}}$ with $\|g_n\|_{\textbf{w}}\le 1$. Fix $w\in T$ and, by replacing $g_n$ with $g_n-g_n(w)$, assume $g_n(w)=0$ for all $n\in \N$. Then, for each $v\in T$, $|g_n(v)|=|g_n(v)-g_n(w)|\le d(v,w)$. Therefore, $g_n$ is uniformly bounded on finite subsets of $T$, and so some subsequence 
 $\{g_{n_k}\}_{k\in\N}$ converges pointwise to some function $g$ on $T$. Fix $\e>0$ and $v\in T^*$. Then, $|g(o)-g_{n_k}(o)|<\displaystyle\frac{\e}{2}$, $|g_{n_k}(v)-g(v)|<\displaystyle{\e}{2|v|}$ and  $|g_{n_k}(v^-)-g(v^-)|<\frac{\e}{2|v|}$ for all $k$ sufficiently large. Thus,
\ben |v|Dg(v)&\le& |v||g(v)-g(v^-)-(g_{n_k}(v)-g_{n_k}(v^-))|+|v|D g_{n_k}(v)\nonumber\\
&<&\varepsilon +|v|D g_{n_k}(v),\nonumber\eeqn
for $k$ sufficiently large.
 Consequently, $g\in\Lip_{\textbf{w}}$ we have
\ben \|g\|_{\textbf{w}} &=&|g(o)|+\sup_{v\in T^*}|v|Dg(v)\nonumber\\
&\le & |g(o)-g_{n_k}(o)|+|g_{n_k}(o)|+\e +
\sup_{v\in T^*}Dg_{n_k}(v)\nonumber\\&<&2\e+\|g_{n_k}\|_{\textbf{w}}\le 2\e+1.\nonumber\eeqn Since $\e$ was arbitrary, it follows that $\|g\|_{\textbf{w}}\le 1$. Therefore, the sequence $\{f_k\}$ defined by $f_k=g_{n_k}-g$ is bounded in $\Lip_{\textbf{w}}$ and converges to 0 pointwise, hence, by the hypothesis, $\|\psi f_k\|_\infty\to 0$ as $n\to\infty$. We conclude that $\psi g_{n_k}\to\psi g$ in $L^\infty$, proving the compactness of $M_\psi$.
\end{proof}

By an analogous argument, we obtain the corresponding compactness criterion for $M_\psi : \Lip_{\textbf{w},0} \to L^\infty$.

\begin{lemma}\label{compact:chara-I} A bounded multiplication operator $M_\psi$ from {\rm $\Lip_{{\textbf{w}},0}$} to $L^\infty$ is compact if and only if for every bounded sequence $\{f_n\}$ in {\rm $\Lip_{{\textbf{w}},0}$} converging to 0 pointwise, the sequence $\|\psi f_n\|_\infty$ approaches 0 as $n\to\infty$.\end{lemma}

\begin{theorem}\label{compactLwLinfty} For a bounded operator $M_\psi$ from {\rm$\Lip_{\textbf{w}}$} to $L^\infty$ (or equivalently from {\rm$\Lip_{{\textbf{w}},0}$} to $L^\infty$) the following statements are equivalent:
\begin{enumerate}
\item[\rm{(a)}] {\rm $M_\psi:\Lip_{\textbf{w}}\to L^\infty$} is compact.
\item[\rm{(b)}] {\rm $M_\psi:\Lip_{{\textbf{w}},0}\to L^\infty$} is compact.
\item[\rm{(c)}] $\lim\limits_{|v|\to \infty}\log|v||\psi(v)|=0.$
\end{enumerate}
\end{theorem}

\begin{proof} (a)$\Longrightarrow$(b) is trivial.

(b)$\Longrightarrow$(c): Let $\{v_n\}$ be a sequence of vertices such that $1\le |v_n|\to\infty$. We need to show that $$\lim\limits_{n\to \infty}\log|v_n||\psi(v_n)|=0.$$
For $n\in\N$ define
$$f_n(v)=\begin{cases} \ \ 0 &\hbox{ if }v=0,\\
\frac{(\log|v|)^2}{\log|v_n|}&\hbox{ if }1\le |v|<|v_n|,\\
\log|v_n|&\hbox{ if }|v|\ge |v_n|.\end{cases}$$
Then $\{f_n\}$ converges to 0 pointwise. Using the fact that $|v|(\log|v|-\log(|v|-1))\le 1$ for any choice of $v$ in $T^*$,  we have
$$|v|Df_n(v)=\frac{|v|\left[(\log|v|)^2-(\log(|v|-1))^2\right]}{\log|v_n|}\le \frac{\log|v|+\log(|v|-1)}{\log|v_n|}\le 2,$$
for $1\le |v|\le |v_n|$. Moreover, $|v|Df_n(v)=0$ for $|v|>|v_n|$. Thus, $f_n\in\Lip_{{\textbf{w}},0}$ and $\{\|f_n\|_{\textbf{w}}\}$ is bounded. By the compactness of $M_\psi$ as an operator acting on $\Lip_{{\textbf{w}},0}$ and by Lemma~\ref{compact:chara-I}, we deduce $$\log|v_n||\psi(v_n)|\le \|\psi f_n\|_\infty\to 0$$ as $n\to\infty$.

(c)$\Longrightarrow$(a): Assume $\{f_n\}$ is a sequence in $\Lip_{\textbf{w}}$ converging to 0 pointwise and such that $a=\sup\limits_{n\in\N}\|f_n\|_{\textbf{w}}<\infty$. By Lemma~\ref{new}, for all $v\in T^*$ and all $n\in \N$, we have
$$|\psi(v)f_n(v)|\le a(1+\log|v|)|\psi(v)|.$$
Fix $\e>0$.  There exists $N\in\N$ such that $N\ge 3$ and for $|v|\ge N$, $\log|v||\psi(v)|<\displaystyle\frac{\e}{2a}.$ Thus, for $|v|\ge N$ and for all $n\in\N$, $|\psi(v)f_n(v)|\le 2a\log|v||\psi(v)|<\e.$
On the other hand, since $f_n\to 0$ pointwise, for each vertex $v$ such that $|v|<N$ and $\psi(v)\ne 0$, we obtain $|f_n(v)|<\displaystyle\frac{\e}{|\psi(v)|}$ for all $n$ sufficiently large. Hence $|\psi(v)f_n(v)|<\e$ for all $v\in T$ and all $n$ sufficiently large. Therefore, $\|M_\psi f_n\|_\infty\to 0$ as $n\to\infty$, which, by Lemma~\ref{compact:chara}, proves the compactness of $M_\psi$.
\end{proof}

Next, we determine the essential norm of the bounded multiplication operators $M_\psi$ from $\Lip_{\textbf{w}}$ or $\Lip_{{\textbf{w}},0}$ to $L^\infty$.

\begin{theorem}\label{essnorm_toinfty} Let $M_\psi$ be a bounded multiplication operator from {\rm$\Lip_{\textbf{w}}$} or {\rm$\Lip_{{\textbf{w}},0}$} to $L^\infty$. Then
$$\|M_\psi\|_e=\lim_{n\to\infty}\sup_{|v|\ge n}\log|v||\psi(v)|.$$
\end{theorem}

\begin{proof} Define $A(\psi)=\lim\limits_{n\to\infty}\sup\limits_{|v|\ge n}\log|v||\psi(v)|.$ If $A(\psi)=0$, then by Theorem~\ref{compactLwLinfty}, $M_\psi$ is compact, hence its essential norm is 0. So assume $A(\psi)>0$. We first show that $\|M_\psi\|_e\ge A(\psi)$. Let $\{v_n\}$ be a sequence in $T$ such that $1\le |v_n|\to\infty$ and $$A(\psi)=\lim_{n\to\infty}\log|v_n||\psi(v_n)|.$$ Fix $p\in (0,1)$ and for each $n\in\N$, define
$$f_{n,p}(v)=\begin{cases} \ \ 0 &\quad\hbox{ if }v=0,\\
\frac{\ (\log|v|)^{p+1}}{(\log|v_n|)^p} &\quad\hbox{ if }1\le|v|<|v_n|,\\
\ \log|v_n| &\quad\hbox{ if }|v|\ge |v_n|.\end{cases}$$
Then $\{f_{n,p}\}$ converges to 0 pointwise, $f_{n,p}\in\Lip_{{\textbf{w}},0}$, $f_{n,p}(v_n)=\log|v_n|$, and
\ben \|f_{n,p}\|_{\textbf{w}}&=&\sup_{2\le|v|\le |v_n|}\frac{|v|}{(\log|v_n|)^p}\left[\left(\log|v|\right)^{p+1}-\left(\log(|v|-1)\right)^{p+1}\right]\nonumber\\
&=&  \frac{|v_n|}{(\log|v_n|)^p}\left[(\log|v_n|)^{p+1}-(\log(|v_n|-1))^{p+1}\right]\le p+1.\nonumber\eeqn
 By Lemma~\ref{weakconv_Lipw}, $\{f_{n,p}\}$ converges to $0$ weakly in $\Lip_{{\textbf{w}},0}$. Let $K$ be a compact operator from $\Lip_{{\textbf{w}},0}$ (or equivalently, from $\Lip_{\textbf{w}}$) to $L^\infty$. Since compact operators are completely continuous, it follows that $\|K f_{n,p}\|_\infty\to 0$ as $n\to\infty$. Thus,
\ben \|M_\psi - K\|&\ge &\limsup_{n\to\infty}\frac{\|(M_\psi-K)f_{n,p}\|_\infty}{\|f_{n,p}\|_{\textbf{w}}}\nonumber\\&\ge &\frac1{p+1}\limsup_{n\to\infty}\|M_\psi f_{n,p}\|_\infty\nonumber\\&\ge&
\frac1{p+1}\limsup_{n\to\infty}\log|v_n||\psi(v_n)|.\nonumber\eeqn
Taking the infimum over all such compact operators $K$ and passing to the limit as $p$ approaches 0, we obtain
$$\|M_\psi\|_e\ge \lim_{n\to\infty}\log|v_n||\psi(v_n)|=A(\psi).$$

To prove the estimate $\|M_\psi\|_e\le A(\psi)$, for each $n\in \N$ and for $f\in \Lip_{\textbf{w}}$, define
$$K_nf(v)=\begin{cases} f(v) &\quad\hbox{ if }|v|\le n,\\
 f(v_n) &\quad\hbox{ if }|v|> n,\end{cases}$$
where $v_n$ is the ancestor of $v$ of length $n$. In the proof of Theorem~\ref{upp_est_ess}, it is was shown that $K_n$ is a compact operator on $\Lip_{\textbf{w}}$. Since $M_\psi:\Lip_{\textbf{w}}\to L^\infty$ is bounded, it follows that $M_\psi K_n$ is also compact as an operator from $\Lip_{\textbf{w}}$ to $L^\infty$.

Let $v\in T$, and let $w$ be a vertex in the path from $o$ to $v$ of length $k\ge 1$. Label the vertices from $w$ to $v$ by $v_j$, $j=k,\dots,|v|$. Then for $f\in\Lip_{\textbf{w}}$ with $\|f\|_{\textbf{w}}=1$, we have
$$|f(v)-f(w)|\le \sum_{j=k+1}^{|v|}|f(v_j)-f(v_{j-1})|\le \sum_{j=k+1}^{|v|}\frac1{j}\le \log|v|.$$ Thus
$$\|(M_\psi-M_\psi K_n)f\|_\infty=\sup_{|v|>n}|\psi(v)||f(v)-f(v_n)|\le \sup_{|v|>n}\log|v||\psi(v)|.$$
We deduce
$$\|M_\psi\|_e \le \sup_{\|f\|_{\textbf{w}}=1}\|(M_\psi-M_\psi K_n)f\|_\infty\le \sup_{|v|>n}\log|v||\psi(v)|.$$
Taking the limit as $n\to\infty$, we obtain $\|M_\psi\|_e \le A(\psi)$.	
\end{proof}

\section{Multiplication operators from $L^\infty$ to $\Lip_{\textbf{w}}$ or $\Lip_{{\textbf{w}},0}$}
In this last section, we study the multiplication operators $M_\psi$ from $L^\infty$ into the weighted Lipschitz space or the little weighted Lipschitz space.  We first characterize the bounded operators and determine the operator norm.  We also show there are no isometries among such operators.  Finally, we characterize the compact multiplication operators and determine the essential norm.

\subsection{Boundedness and Operator Norm}

For a function $\psi$ on $T$, define
$$\eta_\psi=|\psi(o)|+\sup_{v\in T^*}|v|\left[|\psi(v)|+|\psi(v^-)|\right].$$

\begin{theorem}\label{charbound} For a function $\psi$ on $T$, the following statements are equivalent:
\begin{enumerate}
\item[\rm{(a)}] {\rm $M_\psi:L^\infty\to \Lip_{\textbf{w}}$} is bounded.
\item[\rm{(b)}] $\sup\limits_{v\in T}|v||\psi(v)|<\infty$.
\end{enumerate}
Furthermore, under these conditions, we have
$$\|M_\psi\|=\eta_\psi.$$
\end{theorem}

\begin{proof} $(a)\Longrightarrow (b)$: Assume $M_\psi:L^\infty\to \Lip_{\textbf{w}}$ is bounded. Fix $v\in T^*$. Since $\chi_v\in L^\infty$ and $\|\chi_v\|_\infty=1$, the function $\psi \chi_v\in\Lip_{\textbf{w}}$, so
$$|v||\psi(v)|<(|v|+1)|\psi(v)|=\sup_{w\in T^*}|w|D(\psi \chi_v)(w)\le \|M_\psi\|.$$  Thus, $\displaystyle\sup_{v \in T} |v||\psi(v)|$ is finite.

$(b)\Longrightarrow (a)$: Suppose $\displaystyle\sup_{v \in T} |v||\psi(v)| < \infty$. Let $f\in L^\infty$ such that $\|f\|_\infty=1$. Then
$$\|M_\psi f\|_{\textbf{w}}\le |\psi(o)|+\sup_{v\in T^*} |v|\left[|\psi(v)|+|\psi(v^-)|\right]<\infty.$$
Thus, $M_\psi$ is bounded and $\|M_\psi\|\le \eta_\psi$.

	We next show that $\|M_\psi\|\ge \eta_\psi$. The inequality is obvious is $\psi$ is identically 0. For $\psi$ not identically 0 and for $v\in T$, define
$$f(v)=\begin{cases}\ \ \ \ 0 &\quad\hbox{ if }\psi(v)=0,\\
\ \ \overline{\psi(v)}/|\psi(v)|&\quad\hbox{ if }\psi(v)\ne 0, |v| \hbox{ even},\\
-\overline{\psi(v)}/|\psi(v)|&\quad\hbox{ if }\psi(v)\ne 0, |v| \hbox{ odd}.\end{cases}$$
Then $\|f\|_\infty=1$ and for $v\in T^*$, $D(\psi f)(v)=|\psi(v)|+|\psi(v^-)|$, so that
$$ \|M_\psi f\|_{\textbf{w}}=|\psi(o)|+\sup_{v\in T^*}|v|\left[|\psi(v)|+|\psi(v^-)|\right].$$
Thus, $\|M_\psi\|\ge \eta_\psi,$ completing the proof.
\end{proof}

In the next result, we characterize the bounded multiplication operators from $L^\infty$ to $\Lip_{{\textbf{w}},0}$.

\begin{theorem}\label{charbound_0} For a function $\psi$ on $T$, the following statements are equivalent:
\begin{enumerate}
\item[\rm{(a)}] {\rm $M_\psi:L^\infty\to \Lip_{{\textbf{w}},0}$} is bounded.
\item[\rm{(b)}] $\lim\limits_{|v|\to \infty}|v||\psi(v)|=0$.
\end{enumerate}
Furthermore, under these conditions, we have,
$$\|M_\psi\|=\eta_\psi.$$
\end{theorem}

\begin{proof} $(a)\Longrightarrow (b)$: Assume $M_\psi:L^\infty\to \Lip_{{\textbf{w}},0}$ is bounded. Applying $M_\psi$ to the constant function 1, we obtain $\psi=M_\psi 1\in \Lip_{{\textbf{w}},0}$. On the other hand, if $\mathcal{O} = \{v \in T : |v| \text{ is odd}\}$, then $\psi \chi_\mathcal{O} \in \Lip_{{\textbf{w}},0}$, so for $v\in T^*$, we have
\ben |v||\psi(v)|&=&|v||\psi(v)|D\chi_\mathcal{O}(v)\le |v|D(\psi \chi_\mathcal{O})(v)+|v|D\psi(v)|\chi_\mathcal{O}(v^-)|\nonumber\\&\le& |v|D(\psi \chi_\mathcal{O})(v)+|v|D\psi(v)\to 0,\nonumber\eeqn
as $|v|\to\infty$, proving (b).

$(b)\Longrightarrow (a)$: Suppose $|v||\psi(v)|\to 0$ as $|v|\to\infty$. First observe that \ben |v|D\psi(v)&\le& |v||\psi(v)|+\frac{|v|}{|v|-1}(|v|-1)|\psi(v^-)|\nonumber\\&\le& |v||\psi(v)|+2(|v|-1)|\psi(v^-)|\to 0\nonumber\eeqn as $|v|\to\infty$. Then for $f\in L^\infty$ and $v\in T^*$, we have
\ben |v|D(\psi f)(v)&\le& |v||\psi(v)|Df(v)+|v|D\psi(v)|f(v^-)|\nonumber\\
&\le& (2|v||\psi(v)|+|v|D\psi(v))\|f\|_\infty\to 0,\nonumber\eeqn
as $|v|\to\infty$. Thus, $\psi f\in \Lip_{{\textbf{w}},0}$.
The proof of the boundedness of $M_\psi$ and of the formula $\|M_\psi\|=\eta_\psi$  is similar to the case when $M_\psi:L^\infty\to \Lip_{\textbf{w}}$.
\end{proof}

\subsection{Isometries}
As for all other multiplication operators in this article, there are no isometries among the multiplication operators from $L^\infty$ into $\Lip_{\textbf{w}}$ or $\Lip_{\textbf{w},0}$.

Assume $M_\psi$ is an isometry from $L^\infty$ to $\Lip_{\textbf{w}}$ or $\Lip_{{\textbf{w}},0}$. Then, for $v\in T$ the function $f_v=\frac1{|v|+1}\chi_v$ is in $\Lip_{{\textbf{w}},0}$ with $\|M_\psi \chi_v\|_{\textbf{w}}=\|\chi_v\|_\infty=1$. In particular, it follows that $|\psi(o)|=\frac12$ and for $v\in T^*$, $$(|v|+1)|\psi(v)|=1.$$ Thus, $|\psi(v)|=\frac1{|v|+1}$. On the other hand, taking as a test function $f$ the characteristic function of the set $\{v\in T: |v|\le 1\}$, we obtain
$$1=\|f\|_\infty=\|M_\psi f\|_{\textbf{w}}=|\psi(o)|+\max\left\{\sup_{|v|=1}|\psi(v)-\psi(o)|,\sup_{|v|=1}2|\psi(v)|\right\}\ge\frac32,$$ which yields a contradiction. Therefore, we obtain the following result.

\begin{theorem}\label{noisoeither} The are no isometric multiplication operators $M_\psi$ from $L^\infty$ to {\rm $\Lip_{\textbf{w}}$} or {\rm $\Lip_{{\textbf{w}},0}$}.
\end{theorem}

\subsection{Compactness and Essential Norm}

The following two results are compactness criteria for multiplication operators from $L^\infty$ into $\Lip_{\textbf{w}}$ or $\Lip_{\textbf{w},0}$ similar to those given in the previous sections.

\begin{lemma}\label{compact_lemma} A bounded multiplication operator $M_\psi$ from $L^\infty$ to {\rm $\Lip_{\textbf{w}}$} is compact if and only if for every bounded sequence $\{f_n\}$ in $L^\infty$  converging to 0 pointwise, the sequence {\rm$\|\psi f_n\|_{\textbf{w}}$} approaches 0 as $n\to\infty$.\end{lemma}

\begin{proof}  Assume $M_\psi$ is compact and let $\{f_n\}$ be a bounded sequence in $L^\infty$ converging to 0 pointwise. By rescaling the sequence, if necessary, we may assume $\|f_n\|_\infty \le 1$ for all $n\in \N$. By the compactness of $M_\psi$, $\{f_n\}$ has a subsequence $\{f_{n_k}\}$ such that $\{\psi f_{n_k}\}$ converges in the {\rm$\Lip_{\textbf{w}}$}-norm to some function $f\in \Lip_{\textbf{w}}$.
Since by Lemma~\ref{new}, for $v\in T^*$, $$|\psi(v)f_{n_k}(v)-f(v)|\le(1+\log|v|)\|\psi f_{n_k}-f\|_{\textbf{w}},$$ and $|\psi(o)f_{n_k}(o)-f(o)|\le \|\psi f_{n_k}-f\|_{\textbf{w}},$ it follows that $\psi f_{n_k}\to f$ pointwise. Since by assumption, $f_n\to 0$ pointwise, the function $f$ must be identically 0. Thus, the only limit point of the sequence $\{\psi f_n\}$ in $\Lip_{\textbf{w}}$ is 0. Hence $\|\psi f_n\|_{\textbf{w}}\to 0$ as $n \to \infty$.

	Conversely, suppose $\|\psi f_n\|_{\textbf{w}}$ approaches 0 as $n\to\infty$ for every bounded sequence $\{f_n\}$ in $L^\infty$ converging to 0 pointwise. Let $\{g_n\}$ be a sequence in $L^\infty$ with $\|g_n\|_\infty\le 1$. Then some subsequence $\{g_{n_k}\}$ converges to a bounded function $g$. Thus, the sequence $f_{n_k}=g_{n_k}-g$ converges to 0 uniformly and $\|f_n\|_\infty$ is bounded. By the hypothesis, it follows that $\|\psi f_{n_k}\|_{\textbf{w}}\to 0$ as $k\to\infty$. Thus, $\psi g_{n_k}\to\psi g$ in $\Lip_{\textbf{w}}$. Therefore, $M_\psi$ is compact.
\end{proof}

By an analogous argument, we obtain the corresponding result for $M_\psi : L^\infty \to \Lip_{{\textbf{w},0}}$.

\begin{lemma}\label{compact_lemma_I} A bounded multiplication operator $M_\psi$ from $L^\infty$ to {\rm $\Lip_{{\textbf{w}},0}$} is compact if and only if for every bounded sequence $\{f_n\}$ in $L^\infty$  converging to 0 pointwise, the sequence {\rm$\|\psi f_n\|_{\textbf{w}}$} approaches 0 as $n\to\infty$.\end{lemma}

\begin{theorem}\label{compactLinfty} For a bounded operator $M_\psi$ from $L^\infty$ to {\rm $\Lip_{\textbf{w}}$}, the following statements are equivalent:
\begin{enumerate}
\item[\rm{(a)}] $M_\psi$ is compact.
\item[\rm{(b)}] $\lim\limits_{|v|\to \infty}|v||\psi(v)|=0$.
\end{enumerate}
\end{theorem}

\begin{proof} $(a)\Longrightarrow (b)$: Assume $M_\psi$ is compact. Let $\{v_n\}$ be a sequence in $T$ such that $|v_n|\to\infty$ as $n\to\infty$. For $n\in\N$, let $f_n$ denote the characteristic function of the set $\{w\in T: |w|\ge |v_n|\}$. Then  $\|f_n\|_\infty=1$ and $f_n\to 0$ pointwise. By Lemma~\ref{compact_lemma} and the compactness of $M_\psi$, it follows that $$|v_n||\psi(v_n)|=|v_n|D(\psi f_n)(v_n)\le \|M_\psi f_n\|_{\textbf{w}}\to 0$$
as $n\to\infty$.

$(b)\Longrightarrow (a)$: Assume $\lim\limits_{|v|\to \infty}|v||\psi(v)|=0$ and that $\psi$ is not identically 0. In particular, $\psi$ is bounded. Let $\{f_n\}$ be a sequence in $L^\infty$ converging pointwise to 0 and such that $\|f_n\|_\infty$ is bounded above by some positive constant $C$. Then corresponding to $\e>0$, there exists $N\in\N$ such that $|v||\psi(v)|<\displaystyle\frac{\e}{4C}$ for all vertices $v$ such that $|v|\ge N$. Therefore, for $|v|>N$ and $n\in \N$, we have
$$|v|D(\psi f_n)(v)\le |v||\psi(v)|Df_n(v)+|v|D\psi(v)|f_n(v^-)|<\e.$$
Furthermore, the sequence $\{f_n\}$ converges to 0 uniformly on the set $\{v\in T: |v|\le N\}$ so that $|f_n(v)|<\frac{\e}{4N\|\psi\|_\infty}$ for all $n$ sufficiently large. Hence $|v|D(\psi f_n)(v)<\e$ for all $v\in T^*$ and all $n$ sufficiently large. Consequently, $\|\psi f_n\|_{\textbf{w}}\to 0$ as $n\to \infty$. Using Lemma~\ref{compact_lemma}, we deduce that $M_\psi$ is compact.
\end{proof}

	Since the above proof is also valid when $M_\psi$ is a bounded operator from $L^\infty$ to $\Lip_{{\textbf{w}},0}$, through the application of Lemma~\ref{compact_lemma_I}, from Theorems~\ref{charbound_0} and \ref{compactLinfty} we obtain the following result.

\begin{corollary}\label{corcom} For a function $\psi$ on $T$, the following statements are equivalent:
\begin{enumerate}
\item[\rm{(a)}] {\rm $M_\psi:L^\infty\to \Lip_{{\textbf{w}}}$} is compact.
\item[\rm{(b)}] {\rm $M_\psi:L^\infty\to \Lip_{{\textbf{w}},0}$} is bounded.
\item[\rm{(c)}] {\rm $M_\psi:L^\infty\to \Lip_{{\textbf{w}},0}$} is compact.
\item[\rm{(d)}] $\lim\limits_{|v|\to \infty}|v||\psi(v)|=0$.
\end{enumerate}
\end{corollary}

	We now determine the essential norm of the bounded multiplication operators from $L^\infty$ to $\Lip_{\textbf{w}}$.

\begin{theorem}\label{thmess} Let {\rm $M_\psi:L^\infty\to\Lip_{\textbf{w}}$} be bounded. Then
$$\|M_\psi\|_e=\lim_{n\to\infty}\sup_{|v|\ge n}|v|\left[|\psi(v)|+|\psi(v^-)|\right].$$
\end{theorem}

\begin{proof} Set $B(\psi)=\lim\limits_{n\to\infty}\sup\limits_{|v|\ge n}|v|\left[|\psi(v)|+|\psi(v^-)|\right].$ In the case $B(\psi)=0$, then $\lim\limits_{|v|\to\infty}|v|\psi(v)=0$, so by Theorem~\ref{compactLinfty}, $M_\psi$ is compact and thus $\|M_\psi\|_e=0$. So assume $B(\psi)>0$. Then there exists a sequence $\{v_n\}$ in $T$ such that $1\le |v_n|\to\infty$ and
$$B(\psi)=\lim_{n\to\infty}|v_n|\left[|\psi(v_n)|+|\psi(v_n^-)|\right].$$
For each $n\in \N$ let $f_n$ be the function on $T$ defined by
$$f_n(v)=\begin{cases}\ \ \ \  0&\quad\hbox{ if }|v|<|v_n| \hbox{ or }\psi(v)=0,\\
\ \ \overline{\psi(v)}/|\psi(v)| &\quad\hbox{ if }|v|\ge |v_n|, |v| \hbox{ is even, and  }\psi(v)\ne 0,\\
-\overline{\psi(v)}/|\psi(v)| &\quad\hbox{ otherwise.}\end{cases}$$

	Then $\|f_n\|_\infty=1$ and $\{f_n\}$ converges to 0 pointwise. Thus, for any compact operator $K:L^\infty\to\Lip_{\textbf{w}}$, 
 there exists a subsequence $\{f_{n_k}\}$ such that $\|Kf_{n_k}\|_{\textbf{w}}\to 0$ as $k\to\infty$.  Thus
\ben \|M_\psi-K\|&\ge &\limsup_{k\to\infty}\|(M_\psi-K)f_{n_k}\|_{\textbf{w}}\ge \limsup_{k\to\infty}\|\psi f_{n_k}\|_{\textbf{w}}\nonumber\\
&=&\limsup_{k\to\infty}\sup_{|v|\ge|v_{n_k}|}|v|\left[|\psi(v)|+|\psi(v^-)|\right]=B(\psi).\nonumber\eeqn
Therefore, $\|M_\psi\|_e\ge B(\psi).$

We now show that $\|M_\psi\|_e\le B(\psi)$.
 For each $n\in\N$, define the operator $K_n$ on $L^\infty$ by
$$K_nf(v)=\begin{cases} f(v)&\quad\hbox{ if }|v|\le n\\
0&\quad\hbox{ if }|v|> n.\end{cases}$$
Then, for $v\in T^*$, we have
$$|v|D(K_nf)(v)=\begin{cases} |v|Df(v)& \quad\hbox{ for }1\le |v|\le n,\\
(n+1)|f(v^-)|& \quad\hbox{ for }|v|= n+1,\\
0& \quad\hbox{ for }|v|>n+1.\end{cases}$$
Thus, $K_nf\in\Lip_{\textbf{w}}$ with $\|K_nf\|_{\textbf{w}}\le |f(o)|+2n\|f\|_\infty.$

Assume $\{f_k\}$ is a sequence in $L^\infty$ with $\|f_k\|_\infty\le 1$. Then there exists a subsequence $\{f_{k_j}\}$ converging pointwise to some function $f\in L^\infty$. Thus, \small
\ben \|K_nf_{k_j}-K_nf\|_{\textbf{w}} &=&|f_{k_j}(o)-f(o)|\nonumber\\
\  \ \ \ &+&\max\left\{\sup_{1\le |v|\le n}|v|D(f_{k_j}-f)(v),\sup_{|v|=n+1}|v||f_{k_j}(v^-)-f(v^-)|\right\}\nonumber\\&\le &|f_{k_j}(o)-f(o)|\nonumber\\
&+&2n\max\left\{\sup_{1\le |v|\le n}D(f_{k_j}-f)(v),\sup_{|v|=n+1}|f_{k_j}(v^-)-f(v^-)|\right\}.\nonumber\eeqn
\normalsize  So $\|K_nf_{k_j}-K_nf\| \to 0$ as $j\to\infty$. Therefore, $K_n$ is compact, and thus, since $M_\psi$ is bounded, $M_\psi K_n$ is also compact.

For $f\in L^\infty$, we have \small
\ben \|(M_\psi -M_\psi K_n) f\|_{\textbf{w}}&= &\sup_{|v|>n}|v||\psi(v)f(v)-\psi(v^-)f(v^-)+\psi(v^-)K_nf(v^-)|\nonumber\\
 \hskip-20pt&= & \max\left\{\sup_{|v|=n+1}|v||\psi(v)||f(v)|,\sup_{|v|>n+1}|v||\psi(v)f(v)-\psi(v^-)f(v^-)|\right\}\nonumber\\
\hskip-20pt&\le &\sup_{|v|> n}|v|\left[|\psi(v)|+|\psi(v^-)|\right]\|f\|_\infty.\nonumber\eeqn
\normalsize Therefore, we obtain
\ben \|M_\psi\|_e&\le &\limsup_{n\to\infty}\|M_\psi-M_\psi K_n\|\nonumber\\
&=&\limsup_{n\to\infty}\sup_{\|f\|_\infty = 1}\|(M_\psi -M_\psi K_n)f\|_{\textbf{w}}\nonumber\\
&\le & B(\psi),\nonumber\eeqn
thus completing the proof.
\end{proof}

\bibliographystyle{amsplain}
\bibliography{references.bib}
\end{document}